\documentclass[onefignum,onetabnum]{siamart171218}

\usepackage{lipsum}
\usepackage{amsfonts}
\usepackage{graphicx}
\usepackage{epstopdf}
\usepackage{algorithmic}
\ifpdf
  \DeclareGraphicsExtensions{.eps,.pdf,.png,.jpg}
\else
  \DeclareGraphicsExtensions{.eps}
\fi

\newsiamremark{remark}{Remark}
\newsiamremark{hypothesis}{Hypothesis}
\crefname{hypothesis}{Hypothesis}{Hypotheses}
\newsiamthm{claim}{Claim}

\headers{Shape optimization for superconductors}{A. Laurain, M. Winckler, I. Yousept}

\title{Shape optimization for superconductors governed by $\mathbf{H}(\operatorname{\mathbf{curl}})$-elliptic variational inequalities \thanks{Submitted to the editors DATE
\funding{A. Laurain acknowledges the support of FAPESP, process: 2016/24776-6, and of the Brazilian National Council for Scientific and Technological Development  (Conselho Nacional de Desenvolvimento Cient\'ifico e Tecnol\'ogico - CNPq), through the program  ``Bolsa de Produtividade em Pesquisa - PQ 2015'', process: 302493/2015-8.
The work of M. Winckler and I. Yousept was supported by the German Research Foundation Priority Program DFG SPP 1962 "Non-smooth and Complementarity-based Distributed Parameter Systems: Simulation and Hierarchical Optimization", Project YO 159/2-2.}}}

\author{A. Laurain\thanks{Departamento de Matem\'{a}tica Aplicada, Instituto de Matem\'{a}tica e Estat\'{i}stica, Universidade de
\~{S}ao Paulo, Rua do Mat\~{a}o, 1010, CEP 05508-090, \~{S}ao Paulo, SP, Brazil, {\email{laurain@ime.usp.br}}.}
\and M. Winckler\thanks{University of Duisburg-Essen, Fakult{\"a}t f{\"u}r Mathematik, Thea-Leymann-Str. 9, D-45127 Essen, Germany, {\email{malte.winckler@uni-due.de}}, {\email{irwin.yousept@uni-due.de}}.}
\and I. Yousept\footnotemark[3] }

\usepackage{amsopn}

\ifpdf
\hypersetup{
  pdftitle={Shape design for superconductors governed by H(curl)-elliptic variational inequalities},
  pdfauthor={A. Laurain, M. Winckler, I.Yousept}
}
\fi 

\begingroup\expandafter\expandafter\expandafter\endgroup
\expandafter\ifx\csname IncludeInRelease\endcsname\relax
  \usepackage{fixltx2e}
\fi 
\usepackage{hyperref}
\usepackage{cite}
\usepackage{mathtools}
\usepackage{paralist}
\usepackage{enumitem}
\usepackage{graphicx,epstopdf}
\usepackage[caption=false]{subfig}
\usepackage{bm} 
\usepackage{dsfont}
\usepackage{xfrac}

\usepackage{mleftright}
\mleftright

\newenvironment{reduce}
 {\hbox\bgroup\scriptsize$\displaystyle}
 {$\egroup}
  
\newsiamremark{assumption}{Assumption}
\crefname{assumption}{Assumption}{Assumptions}

\DeclareMathOperator{\curl}{\mathbf{curl}}
\DeclareMathOperator{\Div}{div}

\newcommand{\boeta}{\bm{\eta}}
\newcommand{\bzeta}{\bm{\zeta}}

\newcommand{\E}{{\bm E}}   
\newcommand{\e}{{\bm e}}

\renewcommand{\H}{{\bm H}}
\newcommand{\n}{{\bm n}}
\newcommand{\f}{{\bm f}}
\renewcommand{\P}{{\bm P}}
\newcommand{\lam}{{\bm \lambda}}

\newcommand{\Eg}{{\E^\gamma}}
\newcommand{\lamg}{{\lam^\gamma}}
\newcommand{\Pg}{{\P^\gamma}}
\newcommand{\Egn}{{\E^\gamma_n}}
\newcommand{\lamgn}{{\lam^\gamma_n}}

\newcommand{\Lam}{\bm{\Lambda}}
\newcommand{\Lamg}{\bm{\Lambda}_\gamma}
\newcommand{\bxig}{\bm{\psi}^\gamma}

\newcommand{\Egt}{\E^\gamma_t}
\newcommand{\Egz}{\E^\gamma_0}
\newcommand{\Pgt}{\P^\gamma_t}
\newcommand{\Pgz}{\P^\gamma_0}
\newcommand{\Egb}{{\E^\gamma_\star}}

\newcommand{\lamgb}{\lam^\gamma_\star}

\newcommand{\vb}{{\bm v}}
\newcommand{\w}{{\bm w}}

\newcommand{\R}{\mathbb{R}}
\newcommand{\N}{\mathbb{N}}

\newcommand{\V}{{\bm V}}

\newcommand{\Cinf}{\bm{\mathcal C}^\infty_0(\Omega)}
\renewcommand{\L}{{{\bm L}^2(\Omega)}}

\newcommand{\ConematB}{{\mathcal{C}}^1(B, \R^{3\times3})}

\newcommand{\Hcurl}{{\bm H}(\curl)}
\newcommand{\Hzcurl}{{\bm H}_0(\curl)}

\newcommand{\norm}[2]{\|#1\|_{#2}}
\newcommand{\Lnorm}[1]{\norm{#1}{\L}}
\newcommand{\LnormB}[1]{\norm{#1}{{\bm L}^2(B)}}

\newcommand{\Hcurlnorm}[1]{\norm{#1}{\Hcurl}}

\newcommand{\ConematnormB}[1]{\norm{#1}{\mathcal{C}^1(B,\R^{3\times3})}}

\newcommand{\Nedelec}{N{\'e}d{\'e}lec{}}
\newcommand{\Gateaux}{G{\^a}teaux}

\newcommand{\om}{\omega}
\newcommand{\oms}{\om_\star}

\newcommand{\omt}{{{\oms}}}
\newcommand{\omng}{{\om^\gamma_n}}
\newcommand{\omgb}{{\oms^\gamma}}

\newcommand{\T}{\bm{T}_t}
\newcommand{\transp}{\mathsf{T}}
\newcommand{\ML}{\mathcal{L}}
\newcommand{\VV}{\boldsymbol{\theta}}
\newcommand{\M}{\mathds{M}}
\newcommand{\divv}{\operatorname{div}}
\newcommand{\bS}{\boldsymbol{S}}
\newcommand{\bI}{\boldsymbol{I}}

\renewcommand{\O}{\Omega}

\newcommand{\maxg}{{\max}_\gamma}

\newcommand{\ME}{\mathcal{E}}

\newcommand{\One}{\mathds{1}}

\newcommand{\DTME}{\mathcal{K}}

\begin{document}

\maketitle
\begin{abstract}   
This paper is devoted to the theoretical and numerical study of an optimal design problem  in high-temperature superconductivity (HTS).  The shape optimization problem is to find an optimal superconductor shape  which minimizes a certain cost functional under a given   target on the electric field over a specific domain of interest. For the governing PDE-model, we consider an elliptic curl-curl variational inequality (VI) of the second kind with an L1-type nonlinearity.  In particular, the non-smooth VI character and the involved H(curl)-structure make the corresponding shape sensitivity analysis     challenging. To tackle the non-smoothness,  a penalized dual VI formulation    is proposed, leading to the G{\^a}teaux differentiability of the  corresponding dual variable mapping. This property allows us to  derive the distributed shape derivative of the cost functional through rigorous shape calculus   on the basis of   the averaged adjoint method. The developed shape derivative turns out to be uniformly stable with respect to the penalization parameter, and  strong convergence of the penalized problem is guaranteed.  Based on  the achieved theoretical findings, we propose   3D numerical solutions, realised using a level set  algorithm and a Newton method with the   \Nedelec \, edge element discretization. Numerical results indicate a favourable and efficient performance  of the proposed approach for  a specific HTS application in superconducting shielding.  \end{abstract}

\begin{keywords}  shape optimization, high-temperature superconductivity, Maxwell variational inequality, Bean's critical-state model, superconducting shielding, level set method. 
\end{keywords}

\begin{AMS}
35Q93, 35Q60,  49Q10.   
\end{AMS}

\section{Introduction} 
The physical phenomenon of superconductivity is characterized by the zero electrical resistance and the expulsion of   magnetic fields (Meissner effect) occurring up to a certain level of the operating  temperature and magnetic field strength. Nowadays, numerous key technologies  can be realised through high-temperature superconductivity (HTS), including magnetic resonance imaging, magnetic levitation, powerful superconducting wires, particle accelerators, 
magnetic energy storage and many more.  In particular,  to improve and optimize their efficiency and reliability, advanced shape optimization (design) methods are   highly desirable. 

For instance, efficiently designed superconducting shields are a practical way to protect certain areas from magnetic fields. Basically, there are only two possible ways for a magnetic field to penetrate an area shielded by a superconductor -- through the material itself and through opened parts such as holes or gaps. The former depends solely on the properties of the material, the operating temperature, and the magnetic field strength, whereas the latter is also highly affected by the geometry. In the case of an HTS coil for instance, physical experiments \cite{KvitkovicEtAl15} show that the enclosed area is still shielded even if the opened ends are directly facing the field lines. 
On the other hand, if the diameter gets too large, field lines start penetrating the inside. Thus, the following question arises: how should we design superconducting shields in order to save material and still keep the electromagnetic field penetration to a minimum?

 In the recent past,  the Bean critical-state model for HTS has been  extensively studied by several authors. In the eddy current case, it leads to a parabolic Maxwell variational inequality (VI) of the first kind (see \cite{bos94,pri96}), while in the full Maxwell case it gives rise to a hyperbolic Maxwell VI of the second kind (see \cite{Yousept17Hyp,Yousept19Hyp}). For both parabolic and hyperbolic Maxwell VIs,  efficient finite element  methods have been proposed and analyzed  in \cite{ellkas07,barpri15,WincklerYousept18}.

This paper focuses on the sensitivity analysis and numerical investigation for a shape optimization problem in HTS. Our task is to find an admissible superconductor shape   which minimizes a tracking-type objective functional under a given target on the electric field over a specific domain of interest. For the governing PDE-model, we consider the elliptic (time-discrete) counterpart to the Bean critical-state model governed by  Maxwell's equations \cite{WincklerYousept18,Yousept17Hyp,Yousept19Hyp}, given by   an elliptic $\operatorname {\mathbf{curl}}$-$\operatorname {\mathbf{curl}}$ VI of the second kind. To be more precise, let $\Omega \subset \R^3$ be a bounded Lipschitz domain and   
\begin{equation*}
\mathcal{O} := \{\om \subset B  \  \colon \  \om \textrm{ is open, Lipschitz, with uniform Lipschitz constant } L\},
\end{equation*}
with some subset   $B \subset \Omega$. 
For every   admissible superconductor shape $\om\in\mathcal O$, let $\E = \E(\om)\in\Hzcurl$ denote the associated electric field given as the solution of 
\begin{align}\tag{VI\textsubscript{$\om$}}\label{eq:VI}
a(\E,\vb - \E) + \varphi_\om(\vb) -\varphi_\om(\E) \geq \int_\Omega \f \cdot (\vb - \E) \, dx \quad \forall \vb \in \Hzcurl ,
\end{align} 
with the elliptic $\curl$-$\curl$ bilinear form $a\colon \Hzcurl \times \Hzcurl \to \R$ defined by
\begin{align*}
a(\vb,\w) \coloneqq \int_{\Omega}  \nu \curl \vb \cdot \curl \w \, dx + \int_\Omega \varepsilon   \vb    \cdot\w \,dx,
\end{align*}
and the non-smooth $L^1$-type functional
$\varphi_\om \colon  {\bm L}^1(\Omega)  \to \R$, $\vb\mapsto j_c\displaystyle\int_\omega |\vb(x)| \, dx.$
Here, $j_c>0$  denotes the   critical current density of the superconductor $\omega$, and $\epsilon, \nu\colon \Omega \to \R^{3\times3}$ are   the electric permittivity and the magnetic reluctivity, respectively. The right-hand side $\f\colon\Omega \to \R^3$ stands for the applied current source.  Altogether, the optimal HTS design problem  we focus on   reads  as follows:
 \begin{align}\tag{P} \label{eq:P}
\min_{\omega \in \mathcal O} J(\om) :=\frac{1}{2}\int_{B}    \kappa  |\E(\omega)  - \E_d |^2 \, dx  + \int_{\omega}  \, dx,
\end{align}
for some given target $\E_d : B \to \R^3$ and weight coefficient $\kappa: B \to (0,\infty)$. The precise mathematical assumptions for all data involved in \cref{eq:P} are specified  in \cref{assump:MaterialAndData}.

To the best of   authors' knowledge, this paper is the first theoretical and numerical study of  the shape optimization  subject to $\Hcurl$-elliptic VI of the second kind. Both the involved $\Hcurl$-structure and the non-smooth VI character make the corresponding analysis truly challenging. We refer to \cite{you12a,you13,troyou12} for the optimal control of static Maxwell equations. Quite recently, the optimal control of hyperbolic Maxwell variational inequalities arising in HTS was investigated in   \cite{Yousept17OptContr}. While \cref{eq:P} admits an optimal solution (\cref{thm:ExistenceP}), the   differentiability of the   dual variable mapping associated with \cref{eq:VI} cannot be guaranteed. This property is however indispensable for our shape sensitivity analysis.  Therefore, we propose to approximate \cref{eq:P} by replacing \cref{eq:VI}  through its   penalized dual  formulation  \cref{eq:RegVI}, for which the corresponding  dual variable mapping is  G{\^a}teaux-differentiable (\cref{lemma:PropertiesReg}). This allows us to prove our main theoretical result (\cref{thm:sturm}) on the distributed shape derivative of the cost functional through rigorous shape calculus on the basis of the averaged adjoint method. Importantly, the established shape derivative is uniformly stable with respect to the penalization parameter (\cref{thm:ConvergenceShapeDerivative}), and strong convergence of the penalized approach can be guaranteed (\cref{thm:ConvergenceEgLamg}). In addition, the Newton method 
is applicable to the   penalized dual  formulation  \cref{eq:RegVI}. 
Thus, efficient numerical optimal shapes can be realized by means of a level set algorithm along with the developed  shape derivative and a symmetrization strategy.
All these theoretical and numerical evidences indicate  the favourable performance of our approach to deal with shape optimization problems subject to a VI of the second kind. 

Theoretical results on  optimal design problems  were obtained in \cite{BarbuFriedman1991,MR1728769,MR1606879, MR2523581,MR1106360,NeitSokoZole1988,MR1215733,MR3562369}, but there are few early references for VI-constrained numerical shape optimization (see \cite{MR2356899,MR1294835,MR1816853,MR2206676}).
Recent publications include \cite{HintLaur11} regarding a solution algorithm in the infinite dimensional setting for shape optimization problems governed by VIs of the first kind and \cite{MR3878790} concerning a shape optimization method based on a regularized variant of  VI of the first kind.

The concept of shape derivative \cite{DelfourZolesio11,MR3791463,MR1215733} is the basis for the sensitivity analysis of shape functionals. 
We use the \emph{averaged adjoint method}  introduced in \cite{MR3374631}, a Lagrangian-type method  for the efficient computation of shape derivatives.
Lagrangian methods are commonly used in shape optimization and have the advantage of providing the shape derivative without the
need to compute the material derivative of the state (see \cite{MR2033390,MR862783,DelfourZolesio11,MR3350625,MR2166150,MR2434064,MR2977497}). 
Compared to these approaches, the averaged adjoint method is fairly general due to minimal required conditions.

%

\section{Preliminaries}\label{section:Prelim}
For a given Banach space $V$, we denote its norm by $\norm{\cdot}{V}$. If $V$ is a Hilbert space, then $(\cdot,\cdot)_V$ stands for its scalar product and $\|\cdot\|_V$ for the induced norm. In the case of $V = \R^n$, we renounce the subscript in the (Euclidean) norm and write $|\cdot|$. The Euclidean scalar product is denoted by a dot, and $\otimes$ is the standard outer product for vectors in $\R^3$. Hereinafter, a bold typeset indicates vector-valued functions and their respective spaces. 
The Banach space ${\mathcal{C}}^1(\Omega, \R^{3\times 3})$ is equipped with the standard norm, and for $\bm{\mathcal{C}}^{0,1}(\Omega)\coloneqq {\mathcal{C}}^{0,1}(\Omega, \R^3)$ we use
\begin{equation*}
 \norm{\VV}{\bm{\mathcal{C}}^{0,1}(\Omega)} =   \sup_{x\in\Omega} |\VV(x)| + \sup_{x\neq y\in\Omega} \frac{|\VV(x) - \VV(y)|}{|x-y|}.
\end{equation*}
Now, we introduce the central Hilbert space used throughout this paper: 
\begin{align*}
\Hcurl \coloneqq \{\vb\in\L\;:\; \curl\vb\in\L\},
\end{align*}
where $\curl$ is understood in the distributional sense. As usual, $\Cinf$ denotes the space of all infinitely differentiable functions with compact support in $\Omega$. The space $\Hzcurl$ stands for the closure of $\Cinf$ with respect to the $\Hcurl$-norm. 

Next, we present all the necessary assumptions for the material parameters and the given data in \cref{eq:P} and \cref{eq:VI}:
\begin{assumption}[Material parameters and given data]
\label{assump:MaterialAndData} 
\leavevmode
\begin{enumerate}[label = ({A}\arabic*)]
\item \label{assump:P} The subset $B\subset \Omega$ is a Lipschitz domain, $\E_d \in \bm{\mathcal{C}}^{1}(B)$, and $\kappa \in \mathcal{C}^{1}(B)$.
\item \label{assump:Material} We assume $j_c\in \R^+$, and the material parameters  $\epsilon, \nu \colon \Omega \to \R^{3\times3}$ are assumed to be $L^\infty(\Omega, \R^{3\times3})\cap \ConematB$, symmetric   and uniformly positive definite, i.e., there exist $\underline{\nu}, \underline{\epsilon}>0$ such that 
\begin{equation}\label{psd}
  \xi^\transp\nu(x)\xi\geq \underline{\nu}|\xi|^2 
\ \text{and}\    \xi^\transp\epsilon(x)\xi\geq \underline{\epsilon}|\xi|^2
\quad \text{for a.e.}\, \, x\in\O \text{ and all }\xi\in \R^3.
\end{equation}
\item \label{assump:f} The right-hand side satisfies  $\f \in \L\cap\bm{\mathcal{C}}^{1}(B)$. 

\end{enumerate} 
\end{assumption}
\begin{remark}
\leavevmode
\begin{itemize}
 \item[(i)] As pointed out earlier, in the context of superconducting shields, one looks for an optimal superconductor shape $\omega$ that minimizes both the electromagnetic field penetration   and the volume of material.  This can be realised by solving   \cref{eq:P} with $\E_d = 0$ which    obviously satisfies \ref{assump:P}. 
 \item[(ii)]  The material assumption \ref{assump:Material} holds true for instance  in the case of homogeneous HTS material. In this case, $\epsilon,\mu$ are constant in $B$. 
 \item[(iii)] A   choice for the $\f$ satisfying \ref{assump:f} is given by an induction coil away from the superconducting region $B$. In this case,   $\f \equiv 0$ in $B$. 
 \end{itemize}
\end{remark}
For every fixed $\om \subset \mathcal{O}$ the existence of a unique solution $\E\in\Hzcurl$ of \cref{eq:VI} is covered by the classical result \cite[Theorem 2.2]{LionsStampacchia67}, since \ref{assump:Material} implies that the bilinear form $a: \Hzcurl \times \Hzcurl \to \R$ is coercive and continuous. Additionally, it is well-known (cf. \cite{GlowinskiLionsTremolieres81}) that there exists a unique $\lam \in {\bm L}^\infty(\om)$ such that
\begin{equation}\label{eq:LagrangeVI}
\left\{\begin{aligned}
&a(\E,\vb) + \int_\om \lam \cdot \vb \, dx = \int_\O \f\cdot\vb \, dx \quad \forall \vb\in\Hzcurl ,  \\
& |\lam(x)| \leq j_c, \quad \lam(x) \cdot \E(x) = j_c |\E(x)| \text{ for a.e. } x\in \om.
\end{aligned}\right.
\end{equation}
Throughout this paper the following compactness result for the set of domains $\mathcal{O}$ is pivotal to our analysis \cite[Theorem 2.4.10]{MR3791463}. 
\begin{theorem}\label{thm:lipCV}
Let \cref{assump:MaterialAndData} hold and $\{\om_n\}_{n\in\N}\subset\mathcal{O}$. Then, there exist $\om\in\mathcal{O}$ and a subsequence $\{\om_{n_k}\}_{k\in\N}$ which converges to $\om$ in the sense of Hausdorff, and in the sense of characteristic functions. 
Moreover, $\overline{\om}_{n_k}$ and $\partial\om_{n_k}$ converge in the sense of Hausdorff towards $\overline{\om}$ and $\partial\om$, respectively. 
\end{theorem}
With \cref{thm:lipCV} at hand, it is possible to prove existence of an optimal shape for \cref{eq:P} directly. However, as the same result is obtained as a byproduct of \cref{thm:ConvergenceEgLamg}, we do not give a proof at this point.
\begin{theorem}\label{thm:ExistenceP}
 Under \cref{assump:MaterialAndData} the shape optimization problem \cref{eq:P} has an optimal solution $\oms\in\mathcal{O}$.
\end{theorem}

\section{Penalized shape optimization approach} \label{section:RegularizedProblem}
As pointed out earlier, our shape sensitivity analysis requires the differentiability of the dual variable mapping $\E \mapsto \lam$ in $\L$, which cannot be guaranteed in general.
To cope with this regularity issue, we approximate \cref{eq:P} by
\begin{equation}\tag{P\textsubscript{$\gamma$}}\label{eq:Preg}
 \min_{\omega \in \mathcal O} J_\gamma(\om) := \frac{1}{2}\int_{B}   \kappa  |\Eg(\omega)  - \E_d |^2  + \int_{\omega}  \, dx,
 \end{equation}
where $\Eg \coloneqq \Eg(\om) \!\in\! \Hzcurl$ is specified by  the  penalized dual formulation of \cref{eq:LagrangeVI}:
\begin{equation}\label{eq:RegVI}
\left\{\begin{aligned}
&a(\Eg,\vb) + \int_\om \lamg \cdot \vb \, dx = \int_\Omega \f\cdot \vb \, dx \quad\forall \vb\in\Hzcurl\\
&\lamg (x) = \frac{j_c\gamma\Eg(x)}{\maxg\{1,\gamma |\Eg(x)|\}} \text{ for a.e. } x\in \omega.
\end{aligned}\right.
\end{equation}
In this context, $\maxg\colon \R^3 \to \R$ denotes the Moreau-Yosida type regularization (cf. \cite{DeLosReyes11})  of the $\max$-function given by 
\begin{equation}\label{eq:Maxc}
 \maxg \{1, x\} \coloneqq \begin{reduce} 
   \left\{\begin{aligned} 
   & x \quad&&\text{ if } x-1 \geq \frac{1}{2\gamma},\\ 
   & 1 + \frac{\gamma}{2}\left(x - 1 + \frac{1}{2\gamma}\right)^2 \quad &&\text{ if } |x - 1| \leq \frac{1}{2\gamma},\\
   & 1 \quad &&\text{ if } x - 1 \leq -\frac{1}{2\gamma}.
  \end{aligned}\right.
 \end{reduce}
\end{equation}
The following lemma summarizes the \Gateaux-differentiability result for the dual variable mapping associated with \cref{eq:RegVI}:
\begin{lemma}[Theorem 4.1 in \cite{DeLosReyes11}]\label{lemma:PropertiesReg}
  Let $\gamma >0$ and \cref{assump:MaterialAndData} hold. Then,
 \begin{equation}\label{eq:Lamg}
\Lamg \colon \L \to \L, \quad   \Lamg(\e) \coloneqq \frac{j_c\gamma\e}{\maxg\{1,\gamma |\e|\}} 
 \end{equation}
is \Gateaux-differentiable  with the \Gateaux-derivative
\begin{multline}\label{eq:GatDiffLamg}
 \Lamg'(\e) \w = \frac{j_c\gamma\w}{\maxg\{1,\gamma |\e|\}}  \\
 - \gamma \left(\mathds{1}_{\mathcal{A}_\gamma(\e)} 
 + \gamma \left( \gamma|\e| - 1 + \frac{1}{2\gamma}\right) 
 \mathds{1}_{\mathcal{S}_\gamma(\e)}
 \right)\frac{(\e\cdot \w)\Lamg(\e)}{\maxg\{1, \gamma|\e|\}|\e|} \quad \forall \e,\w \in\L,
\end{multline}
where $\One_{\mathcal{A}_\gamma(\e)}$ and $\One_{\mathcal{S}_\gamma(\e)}$ stand for the characteristic functions of the disjoint sets $\mathcal{A}_\gamma(\e) = \left\{ x\in \Omega \,:\, \gamma |\e(x)| \geq 1 + 1\slash2\gamma \right\}$ and $\mathcal{S}_\gamma(\e) = \left\{ x\in\Omega \,:\, |\gamma|\e(x)| - 1 | < 1\slash2\gamma \right\}$, respectively. Furthermore, $\Lamg$ is Lipschitz-continuous and monotone, i.e., 
\begin{equation}\label{eq:MonotoneLam}
(\Lamg(\w_1) - \Lamg(\w_2), \w_1 - \w_2)_\L \geq 0 \quad\forall \w_1,\w_2 \in \L.
\end{equation}
\end{lemma}
In addition to \cref{lemma:PropertiesReg}, it is easy to see that the following estimate holds by definition of $\mathcal{S}_\gamma(\e)$ for every $\e\in\L$:
\begin{equation}\label{eq:SgammaEst}
 \gamma \left( \gamma |\e| -1 + \frac{1}{2\gamma}\right) \leq 1 \quad \text{ a.e. in } \mathcal{S}_\gamma(\e).
\end{equation}
For convenience we define the matrix-valued function $\bxig: \L \to {L}^2(\Omega,\R^{3\times 3})$ by
\begin{align}\label{eq:DefiBxig}
 \bxig(\e) \coloneqq \frac{j_c\gamma \bm{I}_3}{\maxg\{1,\gamma|\e|\}}  - \gamma \left(\One_{\mathcal{A}_\gamma(\e)} + \gamma \bigg(\gamma |\e| - 1 + \frac{1}{2\gamma}\bigg) \One_{\mathcal{S}_\gamma(\e)} \right) \frac{\e \otimes \Lamg(\e)}{\maxg\{1,\gamma|\e|\} |\e|},
\end{align}
where $\bm{I}_3$ denotes the identity matrix in $\R^{3 \times 3}$.
By multiplying \cref{eq:GatDiffLamg} with $\vb\in\L$ and using $(\e\cdot\w)(\Lamg(\e) \cdot \vb)  = \big(\e\otimes\Lamg(\e)\big)\vb\cdot\w$, for all $\e,\vb,\w\in\R^3$, we obtain
\begin{equation}\label{eq:LamgBxig}
 \Lamg'(\e)\w\cdot\vb = \bxig(\e)\vb\cdot\w \quad \forall  \e,\w,\vb\in\L.
\end{equation}

With \cref{lemma:PropertiesReg} at hand, the well-posedness of \cref{eq:RegVI} follows by the theory of monotone operators \cite[p. 40]{Roubicek13}. Moreover, \cref{eq:Maxc} implies for every $\e\in\L$ that
\begin{equation}\label{eq:LamgJc}
 \maxg\{1,\gamma|\e|\} \geq \gamma|\e| \text{ a.e. in } \Omega.
\end{equation}
Applying this estimate to  \cref{eq:Lamg} yields that 
\begin{equation}\label{eq:LamgJc2}
\norm{\Lamg(\e)}{{\bm L}^\infty(\Omega)} \leq j_c \quad \forall \e\in\L.
\end{equation}
Obviously, \cref{eq:Maxc} yields for every $\e\in\L$ that $\maxg\{1,\gamma|\e|\} \geq 1$ almost everywhere in $\Omega$. Hence, we obtain the following estimate for all $\e,\vb,\w \in \L$
\begin{align}\label{eq:BxigEstimate}
\int_\Omega | \bxig(\e)\vb\cdot\w| \,dx &\,\overbrace{\leq}^{\cref{eq:SgammaEst}}   \int_\Omega \frac{j_c\gamma|\vb\cdot\w|}{\maxg\{1,\gamma|\e|\}} \,dx + \gamma\int_\Omega \frac{\big|\big(\e \otimes \Lamg(\e)\big)\vb\cdot\w\big|}{\maxg\{1,\gamma|\e|\} |\e|} \,dx \\\notag
&\overbrace{\leq}^{\cref{eq:LamgJc2}} 2j_c\gamma \Lnorm{\vb}\Lnorm{\w}.
\end{align}
The next result states the existence of an optimal solution to \cref{eq:Preg}.
\begin{theorem}\label{thm:WellPosednessPreg}
 Let \cref{assump:MaterialAndData} hold and $\gamma >0$ be fixed. Then,  \cref{eq:Preg} admits  an optimal shape $\omgb \in \mathcal{O}$. 
\end{theorem}
\begin{proof}
 Let $\{\omng\}_{n\in\N}\subset\mathcal{O}$ be a minimizing sequence for \cref{eq:Preg} with the corresponding states $\Egn \in \Hzcurl$ solving \cref{eq:RegVI} for $\om = \omng$ and $\lamgn \coloneqq \Lam( \Egn)$. 
 Thanks to \cref{thm:lipCV}, there exists a subsequence of $\{\omng\}_{n\in\N}$ (with a slight abuse of notation we use the same index for the subsequence) and $\omgb\subset\mathcal{O}$ such that $\omng \to \omgb$ as $n\to\infty$
in the sense of characteristic functions.

We denote the solution to \cref{eq:RegVI} for $\om = \omgb$ by $\Egb\in\Hzcurl$ and $\lamgb \coloneqq \Lamg(\Egb)$. Now, substracting \cref{eq:RegVI} for $\Egn$ from \cref{eq:RegVI} for $\Egb$ and testing the resulting equation with $\vb = \Egb - \Egn$ yields
\begin{align}\label{eq:CalcConvergenceEngamma} 
&a(\Egb - \Egn,\Egb - \Egn) 
 \overset{\hphantom{\cref{eq:MonotoneLam}}}{=} \int_{\O} (\chi_{\omgb}\lamgb -\chi_{\omng}\lamgn)\cdot ( \Egn - \Egb) \, dx\\\notag
 \overset{\hphantom{\cref{eq:MonotoneLam}}}{=} &\, \int_{\O} (\chi_{\omgb} -\chi_{\omng})\lamgn\cdot (\Egn - \Egb) \, dx - \hspace{-0.8cm}\underbrace{ \int_{\O} \chi_{\omgb}(\lamgb -\lamgn)\cdot (\Egb - \Egn) \, dx}_{=(\Lamg(  \chi_{\omgb} \Egn) - \Lamg(  \chi_{\omgb} \Egb),   \chi_{\omgb} \Egn -   \chi_{\omgb} \Egb)_\L}\\ \notag
\underbrace{\le}_{\cref{eq:MonotoneLam}} &\,\int_{\O} (\chi_{\omgb} -\chi_{\omng})\lamgn\cdot (\Egn - \Egb) \, dx.
\end{align}
 Thus, \cref{eq:CalcConvergenceEngamma} and \ref{assump:Material} of \cref{assump:MaterialAndData}  yield
\begin{align} \label{eq:EstimateEgammaEngamma} \notag
\min\{\underline{\nu}, \underline{\epsilon}\}\Hcurlnorm{\Egb -& \Egn}^2 \leq  \norm{\chi_{\omgb} -\chi_{\omng}}{L^2(\O)} \norm{\lamgn}{{\bm L}^\infty(\O)} \Hcurlnorm{\Egb - \Egn} \\
& \overbrace{\Rightarrow}^{\cref{eq:LamgJc}} \quad \Hcurlnorm{\Egb - \Egn} \leq \frac{j_c}{\min\{\underline{\nu}, \underline{\epsilon}\}}\|\chi_{\omgb} -\chi_{\omng}\|_{L^2(\O)}.
\end{align}
This implies $\Egn  \to \Egb$ in $\Hzcurl$ since $\omng$ converges to $\omgb$ in the sense of characteristic functions
as $n\to\infty$. Hence, we obtain
\begin{align*}   
J_\gamma(\omng) = \frac{1}{2}\int_{B}   \kappa  |\Egn  - \E_d |^2  \,dx
+ \int_{\omng }dx 
\to \frac{1}{2}\int_{B}   \kappa  |\Egb - \E_d |^2 \,dx  + \int_{\omgb}\,dx
= J_\gamma(\omgb).
\end{align*}
Finally, the assertion follows since $\omng$ is a minimizing sequence for \cref{eq:Preg}.
\end{proof}

\section{Shape sensitivity analysis} \label{section:SensitivityAnalysis}

This section is devoted to the sensitivity analysis of the shape functional $J_\gamma(\om)$ in \cref{eq:Preg} for $\gamma>0$ fixed. 
We compute the shape derivative using the averaged adjoint method (see \cite{MR3535238,MR3374631}).
Let $\T  : \Omega\to \Omega$ be the flow of a vector field 
$\VV \in \bm{\mathcal{C}}^{0,1}_c(\Omega,\R^3)$
with compact support in $B$, i.e., $\T(\VV)(X) = x(t,X)$ is the solution to the ordinary differential equation 
\begin{equation}\label{eq:TransODE}
\frac{\text{d}}{\text{d}t} x(t,X) = \VV(x(t,X)) \quad \text{ for } t\in[0,\tau],\quad
x(0,X) = X \in \Omega,
\end{equation}
for some given $\tau >0$.
It is well-known  (see \cite[p. 50]{MR1215733}) that \cref{eq:TransODE} admits a unique solution for a sufficiently small $\tau>0$. Note that $\T(B) = B$ and $\T(X) = X$ for every $X\in\Omega\backslash B$ since $\VV$ has compact support in $B$.
For $\om\in\mathcal O$, we introduce the parameterized family of domains 
$\om_t  := \T(\om)$, for all $t\in[0,\tau]$.
Let us now recall the definition of shape derivative used in this paper.
\begin{definition}[Shape derivative]\label{def1}
Let $K :\mathcal{O} \rightarrow \R$ be a shape functional.
The Eulerian semiderivative of $K$ at $\omega\in\mathcal O$ in direction $\VV \in \bm{\mathcal{C}}^{0,1}_c(\O,\R^3)$
is defined as the limit, if it exists,
\begin{equation*}
d K(\omega)(\VV):= \lim_{t \searrow 0}\frac{K(\omega_t)-K(\omega)}{t},
\end{equation*}
where $\omega_t = \T(\omega)$. Moreover, $K$ is said to be \textit{shape differentiable} at $\omega$ if it has a Eulerian semiderivative at $\omega$ for all $\VV \in \bm{\mathcal{C}}^{0,1}_c(\O,\R^3)$ and the mapping
\begin{align*}
d K(\omega):  \bm{\mathcal{C}}^{0,1}_c(\O,\R^3) \to \R,\quad \VV \mapsto dK(\omega)(\VV)
\end{align*}
is linear and continuous. In this case $d K(\omega)(\VV)$ is called the \textit{shape derivative} at $\omega$.
\end{definition}
In the remainder of this section, we consider the perturbed domain $\om_t$ and denote the corresponding solution of \cref{eq:RegVI} for $\om = \om_t$ by $\Egt \in \Hzcurl$.

\subsection{Averaged adjoint method}\label{subsec:AAM}
We begin by introducing the Lagrangian $\ML:\mathcal{O}\times \Hzcurl\times \Hzcurl\to\R$ associated with \cref{eq:Preg} as follows:
\begin{equation}\label{def:Lag}
\ML(\om,\e,\vb): = \frac{1}{2}\int_{B}  \kappa  |\e - \E_d |^2\,dx  
+ \int_{\om}  \, dx 
+ a(\e,\vb) +\int_{\om} \Lamg(\e) \cdot \vb\,dx - \int_\O  \f \cdot \vb\,dx 
\end{equation}
where $\Lamg$ is given as in \cref{eq:Lamg}. 
In view of \cref{def:Lag}, we have for $\omega\in\mathcal{O}$ and $t\in[0,\tau]$ that
\begin{equation}\label{eq:SaddlePointRepresentation}
 J_\gamma(\om_t) = \ML(\om_t,\Egt,\vb) \quad \forall\vb \in\Hzcurl.
\end{equation}
Moreover, as $\ML$ is linear in $\vb$, the problem of finding $\e\in\Hzcurl$ such that
\begin{align*}\label{eq:LagrState}
\partial_{\vb}\ML(\om_t,\e,\vb;\hat\vb)
= a(\e,\hat\vb) +\int_{\om_t} \Lamg(\e) \cdot \hat\vb\,dx - \int_\O  \f \cdot \hat\vb\,dx = 0 \quad\forall \hat\vb\in  \Hzcurl
\end{align*}
is equivalent to \cref{eq:RegVI} with $\om=\om_t$ and admits the same unique solution $\Egt \in\Hzcurl$.
In order to pull back the integrals over $\om_t$ to the reference domain $\om$, one uses the change of variables $x\mapsto \T(x)$.
Furthermore, to avoid the appearance of the composed functions $\e\circ\T$ and $\vb\circ\T$  due to this  change of variables, we reparameterize the Lagrangian using the following covariant transformation, which is known to be a bijection for ${\bm H}_0(\curl)$ (cf. \cite[p. 77]{Monk03}).
\begin{equation}\label{eq:Transform}
\Psi_t \colon \Hzcurl \to \Hzcurl, \qquad \Psi_t (\e) \coloneqq (D\T ^{-\transp}\e)\circ\T ^{-1}.
\end{equation}
Here $D\T\colon\R^3 \to \R^{3\times3}$ stands for the Jacobian matrix function of $\T$ and we denote $D\T^{-\transp} \coloneqq \big(D\T^{-1}\big)^\transp$. It satisfies the important identity (see \cite[Lemma 11]{MR1662168})
\begin{equation}\label{eq:RelCurl}
\big(\curl  \Psi_t (\e )\big)\circ \T  =  \xi(t )^{-1} D\T \curl \e, 
\end{equation}
with $\xi(t) := \det D\T$.
In this paper we always assume $\tau > 0$ small enough such that  $\xi(t) > 0$ for every $t\in[0,\tau]$. That is, the transformation $\T$ preserves orientation. 
In view of the above discussion, we introduce the {\it shape-Lagrangian}  $G:[0,\tau]\times \Hzcurl\times \Hzcurl \rightarrow \R$ as
\begin{multline}\label{eq:DefiShapeLagr}
G(t,\e,\vb)\coloneqq\mathcal{L}(\om_t,\Psi_t (\e),\Psi_t (\vb)) =  \frac{1}{2}\int_{B}  \kappa  |\Psi_t(\e) - \E_d |^2 \,dx  + \int_{\omega_t}  \, dx\\
+ a(\Psi_t(\e),\Psi_t(\vb)) +\int_{\omega_t} \Lamg(\Psi_t(\e)) \cdot \Psi_t(\vb)\,dx - \int_\O  \f \cdot \Psi_t(\vb)\,dx.
\end{multline}
The change of variables $x\mapsto \T(x)$ inside the integrals \cref{eq:Transform,eq:RelCurl} yields
{\small \begin{align} \notag
 G(t,\e,\vb)  = \frac{1}{2}\int_{B}  \kappa\circ\T  |D\T^{-\transp}\e - \E_d \circ\T |^2\xi(t)\,dx + \int_{\omega} \xi(t)\,dx +   \int_{\Omega} \M_1(t) \curl\e  \cdot \\  \label{eq:Gshort}
   \curl\vb +\M_2(t) \e\cdot \vb\,dx+\int_{\omega}  \M_3(t,\e) \cdot  \vb\,dx - \int_\O  (\f \circ\T)  \cdot   (D\T^{-\transp}\vb)  \xi(t)\,dx, 
\end{align}}
with the notations
$\M_1(t)  : = \xi(t )^{-1} D\T^\transp (\nu\circ\T)  D\T$, 
$\M_2(t)   : = \xi(t) D\T^{-1} (\varepsilon\circ\T) D\T^{-\transp}$ and  
$\M_3(t,\e)   : = \xi(t)  D\T^{-1} \Lamg( D\T^{-\transp}\e)$.
Note that the problem of finding $\e_t\in\Hzcurl$ such that $\partial_\vb G(t,\e_t,0;\hat\vb) = 0$ for all $\hat\vb \in \Hzcurl$ is equivalent to \cref{eq:RegVI} with $\omega=\omega_t$ after applying the change of variables $x\mapsto \T(x)$. Hence, it has the same unique solution $\Egt\in \Hzcurl$.

Next, the shape derivative of $J_\gamma$ is obtained as the partial derivative with respect to $t$ of the shape-Lagrangian $G$ given by \cref{eq:Gshort}. 
For the convenience of the reader, we recall the main result of the averaged adjoint method, adapted to our case. A proof can be found in \cite[Theorem 2.1]{MR3535238} (cf. \cite{MR3374631}).
\begin{theorem}[Averaged adjoint method]\label{thm:AAE}
 Let $\gamma>0$. Moreover, we assume that there exists $\tau \in (0,1]$ such that for every $(t,\vb)\in [0,\tau] \times\Hzcurl$
 \begin{enumerate}[label = \textup{({H}\arabic*)}] 
\item\label{H1} the mapping $[0,1]\ni s\mapsto G(t,s\Egt+(1-s)\Egz,\vb)$ is absolutely continuous;
\item\label{H2} the mapping $[0,1]\ni s\mapsto \partial_\e G(t,s\Egt+(1-s)\Egz,\vb;\hat{\e})$ belongs to $L^1(0,1)$ for every $\hat{\e}\in \Hzcurl$;
\item\label{H3} there exists a unique $\Pgt \in\Hzcurl$ that solves the averaged adjoint equation
\begin{equation}\label{averated_adj}
\int_0^1 \partial_\e G(t,s\Egt+(1-s)\Egz,\Pgt;\hat{\e})\, ds =0 \quad \forall \hat{\e}\in \Hzcurl;
\end{equation}
\item\label{H4} the family $\{\Pgt\}_{t\in [0,\tau]}$  satisfies
\begin{equation}\label{eq:DifferenceQuotientAvAdj}
 \lim_{t\searrow 0} \frac{G(t,\Egz,\Pgt)-G(0,\Egz,\Pgt)}{t}=\partial_tG(0,\Egz,\Pgz). 
\end{equation}
\end{enumerate}
Then, $J_\gamma$ is shape-differentiable in the sense of \cref{def1} and it holds that
 \begin{equation*}
 dJ_\gamma(\om)(\VV) =  \frac{d}{dt} J_\gamma(\om_t)|_{t=0} = \partial_t G(0,\Egz,\Pgz),
 \end{equation*}
where $\Pgz$ is the so-called adjoint state solution of \cref{averated_adj} with $t=0$.
\end{theorem}
 We verify that \ref{H1}--\ref{H4} are satisfied so that we may apply \cref{thm:AAE}.
\begin{lemma}\label{lemma:H0}
Let \cref{assump:MaterialAndData} be satisfied. Then,   \ref{H1} and \ref{H2} hold for every $(t,\vb)\in [0,1] \times\Hzcurl$.
\end{lemma}
\begin{proof}
First of all,  \ref{H1} is a direct consequence of  \cref{eq:Gshort} and \cref{lemma:PropertiesReg}. Before we proceed to prove \ref{H2}, let us introduce the notation $\ME(s):= s\Egt+(1-s)\Egz$. 
Now, fix $\tau \in (0,1]$ and $(t,\vb)\in [0,\tau]\times \Hzcurl$. Thanks to the \Gateaux-differentiability of $\Lamg$ (\cref{lemma:PropertiesReg}) and using \cref{eq:Gshort}, we may compute
\begin{align}\label{eq:DeG}
\partial_\e G(t,\ME(s),\vb;\hat{\e}) =  \int_{B}   \kappa\circ\T  \big(D\T^{-\transp}\hat\e \cdot (D\T^{-\transp}\ME(s) - \E_d \circ\T )\big)\xi(t)\,dx\\\notag
  +\int_{\Omega} \M_1(t) \curl\hat\e\cdot  \curl\vb  +\M_2(t) \hat\e\cdot \vb\,dx  +\int_{\omega}  \partial_\e\M_3\big(t,\ME(s)\big)\hat\e \cdot  \vb \,dx
\end{align}
for every $\hat\e\in\Hzcurl$, where 
\begin{align}\label{eq:M3Deriv}
\int_\om\partial_\e\M_3\big(t,\ME(s)\big)\hat\e\cdot\vb\,dx \overset{\phantom{\cref{eq:LamgBxig}}}{=}\quad\, &\int_\om\xi(t)  D\T^{-1} \Lamg'\big(D\T^{-\transp} \ME(s)\big)(D\T^{-\transp} \hat\e) \cdot \vb \,dx\\\notag
\overbrace{=}^{\cref{eq:DefiBxig}\&\cref{eq:LamgBxig}} &\int_\om \xi(t) D\T^{-2}\bxig\big(D\T^{-\transp} \ME(s)\big)\vb\cdot  \hat\e\,dx,
\end{align}
Moreover,  the following asymptotic expansions hold (see \cite[Lemma 2.31]{MR1215733}) 
\begin{equation}
 \label{eq:ExpXi} \xi(t)  = 1 + t\Div(\VV) + o(t), \
 D\T = \bI_3 +tD\VV +o(t),\
 D\T^{-1} = \bI_3 - tD\VV + o(t)
\end{equation}
such that $o(t) \slash t \to 0$ as $t \to 0$ with respect to $\norm{\cdot}{{\mathcal{C}}(\Omega)}$ and $\norm{\cdot}{{\mathcal{C}}(\Omega,\R^{3\times3})}$, respectively. Hence, \cref{eq:ExpXi} imply that there exists a constant $C>0$ only dependent on $\VV$ such that
\begin{equation}
 \label{eq:EstXi}\norm{\xi(t)}{L^\infty(\Omega)} + \norm{D\T}{L^\infty(\Omega,\R^{3\times3})} + \norm{D\T^{-1}}{L^\infty(\Omega,\R^{3\times3})}  \leq 1 + C\tau.
\end{equation}
Applying \cref{eq:EstXi} in \cref{eq:M3Deriv} leads to
\begin{align}\label{eq:ContinuityM3}
\left|\int_{\omega}  \partial_\e\M_3(t,\ME(s))\hat\e \cdot  \vb\,dx\right|& \;\;\leq\;\;  (1 + C\tau)^3\int_\om \big|\bxig\big(D\T^{-\transp}\ME(s)\big)\vb\cdot\hat\e\big|\,dx \\\notag
&\overbrace{\leq}^{\cref{eq:BxigEstimate}} 2j_c\gamma (1 + C\tau)^3 \Lnorm{\hat\e}\Lnorm{\vb} \quad \forall s \in (0,1).
\end{align}
Thus, the mapping $s\mapsto\int_{\omega}  \partial_\e\M_3(t,\ME(s))\hat\e \cdot  \vb\,dx$ belongs to $L^\infty(0,1) \subset L^1(0,1)$.
In a similar way, since $t\in[0,\tau]$ and $\gamma >0$ are fixed, \cref{eq:EstXi} and \ref{assump:P} of \cref{assump:MaterialAndData}  yield
\begin{align}\label{eq:EstimateDerivRHS}
  &  \int_{B}  \big| \kappa\circ\T \big( D\T^{-\transp}\hat\e \cdot D\T^{-\transp}\ME(s)\big) \xi(t)\big| \,dx \\\notag
  &\hspace{1cm}\leq \, (1 + C\tau)^3 \norm{\kappa}{\mathcal{C}(\Omega)} \Lnorm{\hat\e}\Lnorm{\ME(s)}\\\notag
  &\hspace{1cm}\leq \, (1 + C\tau)^3 \norm{\kappa}{\mathcal{C}(\Omega)} \Lnorm{\hat\e} \big( \Lnorm{\E_0^\gamma} + s \Lnorm{\Egt - \E_0^\gamma} \big)\\\notag
  &\hspace{1cm}\leq\, (1 + s)(1 + C\tau)^3 \norm{\kappa}{\mathcal{C}(\Omega)} \Lnorm{\hat\e}\big(\Lnorm{\Egt - \E_0^\gamma} + \Lnorm{\E_0^\gamma}\big).
  \end{align}
As the remaining terms in \cref{eq:DeG} are independent of $s$, \cref{eq:ContinuityM3,eq:EstimateDerivRHS} imply that the mapping $s\mapsto \partial_\e G(t,\ME(s),\vb;\hat{\e})$ belongs to $L^1(0,1)$ for all $\hat{\e}\in \Hzcurl$ and  $(t,\vb)\in [0,\tau]\times \Hzcurl$. Thus, the proof is complete.
\end{proof} 
\begin{lemma}\label{lemma:H1}
 Let \cref{assump:MaterialAndData} hold. Then, there exists $\tau \in (0,1]$ such that \ref{H3} is satisfied for every $t\in[0,\tau]$. 
 Moreover, \ref{H4} holds as well.
\end{lemma}
\begin{proof}
Fix some arbitrary $\tau>0$ and denote $\ME(s) \coloneqq s\Egt + (1-s)\Egz$ for $s\in(0,1)$. 
 Let $\tau \in (0,1]$ be arbitrarily fixed. In the following, if necessary, we shall reduce $\tau  \in (0,1]$ step by step to prove our result. Let  $t\in[0,\tau]$  and $\hat\e \in\Hzcurl$. 
Thanks to \cref{lemma:H0}, the left-hand side of \cref{averated_adj} is well-defined, and our goal is to  prove the existence of a unique  $\Pgt\in\Hzcurl$ satisfying \cref{averated_adj}. In view of \cref{eq:DeG}, we note that \cref{averated_adj} can be written as
\begin{equation}\label{eq:AveragedAdjoint}
B_t(\Pgt,\hat\e) = F_t(\hat\e) \quad \forall\hat{\e}\in \Hzcurl
\end{equation}
with $B_t\colon\Hzcurl\times\Hzcurl\to \R$ and $F_t\colon\Hzcurl\to\R$ defined by
\begin{align*}
 B_t(\vb, \hat\e) &\coloneqq \int_{\Omega}\! \M_1(t)\! \curl\hat\e\cdot  \curl\vb  +\M_2(t) \hat\e\cdot \vb\,dx + \int_0^1\!\! \int_{\omega} \! \partial_\e\M_3(t,\ME(s))\hat\e\cdot  \vb \, dx \,ds,\\
 F_t(\hat\e) &\coloneqq - \int_0^1 \int_{B}   \kappa\circ\T  \left(D\T^{-\transp}\hat\e \cdot \left(D\T^{-\transp}\ME(s) - \E_d \circ\T \right)\right)\xi(t)\,dx\,ds.
\end{align*}
Thanks to \ref{assump:Material} and \cref{eq:EstXi,eq:ContinuityM3}, $B_t$ is a bounded bilinear form. 
In order to apply the Lax-Milgram lemma, we have to prove the coercivity of $B_t$. 
The asymptotic expansions \cref{eq:ExpXi} show that $\M_1(t)$ and $\M_2(t)$ are small perturbations of $\nu$ and $\epsilon$, respectively. Thus,   if necessary, we may reduce the number $\tau \in (0,1]$ such that, in view of \cref{psd}, $\M_1(t)$ and $\M_2(t)$ are uniformly positive definite  for all    $t\in[0,\tau]$ with:
\begin{align}\label{eq:CoercM1M2}
 \int_{\Omega} \M_1(t) \curl\vb\cdot\curl\vb  +\M_2(t) \vb\cdot \vb\,dx \geq C_1 \Hcurlnorm{\vb}^2 \quad \forall\vb\in\Hzcurl,
\end{align}
for some constant $C_1>0$ depending  only on $\VV,\epsilon$ and $\nu$. In order to keep the notation short, let us define $\DTME(s) \coloneqq D\T^{-\transp}\ME(s) \in\Hzcurl$ as well as the sets ${\mathcal{A}}_\gamma(s) \coloneqq \mathcal{A}_\gamma(\DTME(s)) \subset \Omega$ and ${\mathcal{S}}_\gamma(s) \coloneqq \mathcal{S}_\gamma(\DTME(s)) \subset \Omega$ for $s\in (0,1)$ (cf. \cref{lemma:PropertiesReg}). 
We estimate the third term in $B_t$ which, in view of \cref{eq:DefiBxig,eq:M3Deriv}, corresponds to
\begin{multline}\label{eq:ToEstimate}
\int_0^1\int_\om \partial_e \M_3(t, \ME(s))\vb\cdot \vb\,dx\,ds= \int_0^1\int_\om \xi(t) D\T^{-2} \left[\frac{j_c\gamma \bm{I}_3}{\maxg\{1,\gamma|\DTME(s)|\}} \right. \\ 
        \left. - \gamma \left(\One_{\mathcal{A}_\gamma(s)} + \gamma \bigg(\gamma |\DTME(s)| - 1 + \frac{1}{2\gamma}\bigg) \One_{\mathcal{S}_\gamma(s)} \right) \frac{\DTME(s) \otimes \Lamg(\DTME(s))}{\maxg\{1,\gamma|\DTME(s)|\} |\DTME(s)|}\right]\vb\cdot\vb \, dx\,ds.
\end{multline}
Therefore, we fix $s\in(0,1)$ and estimate the three summands in \cref{eq:ToEstimate} separately. 
We begin with the first term and note that \cref{eq:ExpXi} implies,  possibly after reducing $\tau>0$,  that there exists a constant $C>0$, depending only on $\VV$, such that
$ \xi(t) \geq 1 - C\tau>0$, and  $D\T^{-2}\bm{\eta}\cdot \bm{\eta} \geq (1 - C\tau)^2 |\bm{\eta}|^2$ for all $\bm{\eta}\in\R^3$
and almost everywhere in $\Omega$. 
Hence,
\begin{equation}\label{eq:EstPositivePart}
 \int_\om j_c\gamma\xi(t)\frac{ D\T^{-2} \vb\cdot\vb}{\maxg\{1,\gamma|\DTME(s)|\}}\,dx \geq (1 - C\tau)^3 \int_\om \frac{j_c\gamma|\vb|^2}{\maxg\{1,\gamma|\DTME(s)|\}}\,dx.
\end{equation}
Now, we proceed to estimate the integrals over the disjoint sets $\om\cap \mathcal{A}_\gamma(s)$ and $\om\cap\mathcal{S}_\gamma(s)$ appearing in the last two summands in \cref{eq:ToEstimate}. 
We obtain
\begin{align}\label{eq:EstAgammaPart}
 &\left|\int_{\om \cap \mathcal{A}_\gamma(s)} \gamma  \xi(t)D\T^{-2}\frac{\DTME(s) \otimes \Lamg(\DTME(s))\vb\cdot \vb}{\maxg\{1,\gamma|\DTME(s)|\} |\DTME(s)|} \,dx\right|\\\notag
 \overbrace{\leq}^{\cref{eq:Lamg}\&\cref{eq:LamgJc2}}\,
    &\norm{\xi(t)}{L^\infty(\Omega)}\norm{D\T^{-1}}{L^\infty(\Omega,\R^{3\times3})}^2 \int_{\om\cap\mathcal{A}_\gamma(s)} \frac{j_c\gamma|\vb|^2}{\maxg\{1,\gamma|\DTME(s)|\}}  \,dx\\\notag
 \overbrace{\leq}^{\cref{eq:EstXi}} 
     &(1 + C\tau)^3\int_{\om\cap\mathcal{A}_\gamma(s)} \frac{j_c\gamma|\vb|^2}{\maxg\{1, \gamma|\DTME(s)|\}} \,dx.
\end{align}
For the last summand,  we use the same arguments and also \cref{eq:SgammaEst} to deduce
\begin{multline}\label{eq:EstSgammaPart}
  \left|\int_{\om\cap\mathcal{S}_\gamma(s)} \gamma^2 \left(\gamma |\DTME(s)| - 1 + \frac{1}{2\gamma}\right) \xi(t)  D\T^{-2}  \frac{\DTME(s) \otimes \Lamg(\DTME(s))\vb\cdot \vb}{\maxg\{1,\gamma|\DTME(s)|\}\, |\DTME(s)|} \,dx\right|\\
 \leq (1 + C\tau)^3\int_{\om\cap \mathcal{S}_\gamma(s)} \frac{j_c\gamma|\vb|^2}{\maxg\{1, \gamma|\DTME(s)|\}}\,dx.
\end{multline}
Note that the constant $C>0$  in \cref{eq:EstPositivePart,eq:EstAgammaPart,eq:EstSgammaPart} is the same in the three inequalities. 
Thus, we sum up \cref{eq:EstAgammaPart,eq:EstSgammaPart} and substract the result from \cref{eq:EstPositivePart} to obtain
\begin{align*}
 \int_\om \partial_e \M_3(t, \ME(s)) \vb\cdot \vb \,dx 
%
 &\geq \big(1 + 3(C\tau)^2\big)\int_{\om \backslash (\mathcal{A}_\gamma(s) \cup \mathcal{S}_\gamma(s))} \frac{j_c\gamma |\vb|^2}{\maxg\{1, \gamma|\DTME(s)|\}}\,dx \\\notag
                &\quad - ( 6C\tau + 2(C\tau)^3 ) \int_{\omega} \frac{ j_c\gamma|\vb|^2}{\maxg\{1, \gamma|\DTME(s)|\}}\,dx.
\end{align*}
As the first term is non-negative and $\maxg\{1, \gamma |\DTME(s)|\} \geq 1$, we conclude for \cref{eq:ToEstimate} that
\begin{align}\label{eq:CoercM3}
 \int_0^1\int_\om \partial_e \M_3(t, \ME(s))\vb\cdot \vb \,dx\,ds \geq   - ( 6C\tau + 2(C\tau)^3 )  j_c\gamma \norm{\vb}{{\bm L}^2(\om)}^2.
\end{align}
The coercivity of $B_t$ follows, as \cref{eq:CoercM1M2} in combination with \cref{eq:CoercM3} implies   that
\begin{align}\label{coerB}
 B_t(\vb,\vb) \geq  \underbrace{(C_1  -   6C\tau - 2(C\tau)^3 )}_{\eqqcolon C_2 } \Hcurlnorm{\vb}^2  \quad\forall\vb\in\Hzcurl.
\end{align}
If necessary, we  further reduce $\tau \in (0,1]$ such that  $C_2>0$ holds true. In turn, for all $t\in[0,\tau]$, $B_t$ is coercive with the coercitivity constant $C_2>0$, independent of $t$. Ultimately, the Lax-Milgram lemma yields the existence of a unique solution $\Pgt\in\Hzcurl$ of the averaged adjoint equation \cref{averated_adj}. Thus, \ref{H3} holds.

We finish this proof by verifying \ref{H4}.  To this aim, let $\{t_k\}_{k \in \N} \subset  (0,\tau]$ be a null sequence. 
First of all,    the sequence    $\{\E^\gamma_{t_k}\}_{k \in \N} \subset \Hzcurl$ of solutions  to the perturbed state equations  \cref{eq:RegVI} with $\om = \om_{t_k}$ is   bounded. 
This follows readily by inserting $\vb=\E^\gamma_{t_k}$ into  \cref{eq:RegVI} which yields
\begin{multline}\label{eq:BoundednessEgt}
 \min(\underline{\nu},\underline{\epsilon})\Hcurlnorm{\E^\gamma_{t_k}}^2 \leq a(\E^\gamma_{t_k},\E^\gamma_{t_k}) 
 \leq (\Lnorm{\f} + j_c) \Hcurlnorm{\E^\gamma_{t_k}} \\
 \quad\Rightarrow \quad \Hcurlnorm{\E^\gamma_{t_k}} \leq \min(\underline{\nu},\underline{\epsilon})^{-1} (\Lnorm{\f} + j_c) \quad \forall k \in \N.
\end{multline}
Hereafter, we deduce a similar estimate for $\{\P^\gamma_{t_k}\}_{k \in \N}$ by testing \cref{eq:AveragedAdjoint} with $\hat\e = \P^\gamma_{t_k}$ and using  \cref{coerB} along with \cref{eq:EstXi}:\begin{multline}\label{eq:BoundednessPgt}
C_2\|\P^\gamma_{t_k}\|_{\Hzcurl }^2 \leq B_t(\P^\gamma_{t_k},\P^\gamma_{t_k}) = F_t(\P^\gamma_{t_k})\\ \leq {\|\kappa\|_{\mathcal C(\Omega) }(1+C\tau)^3}(\Lnorm{\E^\gamma_{t_k}} + \Lnorm{\Egz} +  \|\E_d\|_\L) \Lnorm{\P^\gamma_{t_k}} \quad \forall k \in \N.
\end{multline}
Since the constant $C_2$ and $C$ are independent of $k\in \N$, the above estimate implies the boundedness of $\{\P^\gamma_{t_k}\}_{k \in \N} \subset \Hzcurl$. Hence,   there exists a subsequence $\{t_{k_j}\}_{j\in\N}\subset\{t_k\}_{k\in\N}$  converging weakly in $\Hzcurl$ to some  $\P^\star\in\Hzcurl$. 
By \cref{eq:ExpXi} and as the solution of \cref{eq:AveragedAdjoint} is unique, passing to the limit $t=t_{k_j}\to 0$ in \cref{eq:AveragedAdjoint} yields $\P^\star = \Pgz$. 
Since $\Pgz$ is independent of  the choice of the subsequence $\{t_{k_j}\}_{j\in\N}$, a standard argument implies  the weak convergence of the whole sequence:
\begin{equation}\label{weakconP}
\P^\gamma_{t_{k}} \rightharpoonup  \P^\star \quad  \textrm{weakly in } \Hzcurl \quad \textrm{as } k \to \infty.
\end{equation}
  Let us now consider the differential quotient
\begin{align}\notag
\frac{G(t_k,\Egz,\P^\gamma_{t_{k}})-G(0,\Egz,\P^\gamma_{t_{k}})}{t_k} 
 =  \int_{B}  \frac{\M_0(t_k) - \M_0(0)}{t_k}\,dx + \int_{\omega} \frac{\xi(t_k) - \xi(0)}{t_k}\,dx \\\notag
 +\int_{\Omega} \frac{\M_1(t_k) - \M_1(0)}{t_k} \curl\Egz\cdot  \curl\P^\gamma_{t_{k}} + \frac{\M_2(t_k) - \M_2(0)}{t_k} \Egz\cdot \P^\gamma_{t_{k}}\,dx \\ 
\label{eq:DifferenceQuotientFormula}  +\int_{\omega}  \frac{\M_3(t_k,\Egz) - \M_3(0,\Egz)}{t_k} \cdot \P^\gamma_{t_{k}}\,dx - \int_\O  \frac{\M_4(t_k) - \M_4(0)}{t_k} \cdot\P^\gamma_{t_{k}}\,dx,
\end{align}
with
$\M_0(t_k):= \frac{1}{2} \kappa\circ\bm{T}_{t_k}  |D\bm{T}_{t_k}^{-\transp}\Egz - \E_d \circ \bm{T}_{t_k} |^2\xi(t_k)$ and $
\M_4(t_k) := \xi(t_k) D\bm{T}_{t_k}^{-1} (\f \circ\bm{T}_{t_k})$.
First, \cref{eq:ExpXi} yields the strong convergence  
\begin{equation}\label{eq:Xiprime}
\lim_{k \to \infty}\frac{\xi(t_k) - \xi(0)}{t_k} = \divv\VV \quad \textrm{in } \mathcal C(\Omega).
\end{equation}
Moreover, thanks to \cref{assump:MaterialAndData}, \cref{eq:ExpXi} and  $\operatorname{supp} \VV \subset\subset B$,  we obtain the strong convergence of    $\big(\M_i(t_k) - \M_i(0) \big) \slash t_k$, $i = 0,1,2,4$, as $k \to \infty$ in $L^\infty(\O)$:
\begin{align}\label{eq:M0prime}
\lim_{k \to \infty}\frac{ \M_0(t_k)  -  \M_0(0)}{t_k} &= \frac{1}{2}(\widetilde{\nabla\kappa}\cdot\VV + \kappa\Div\VV) |\Egz - \E_d|^2 \\ \notag &\qquad- \kappa(\Egz - \E_d)\cdot(D\VV^\transp\Egz - \widetilde{D\E_d}\VV)  \\
 \label{eq:M1prime} \lim_{k \to \infty}\frac{ \M_1(t_k)  -  \M_1(0)}{t_k}   & = - (\divv\VV) \nu + D\VV^\transp \nu  +\nu D\VV + \widetilde{D\nu}\VV , \\
 \label{eq:M2prime} \lim_{k \to \infty}\frac{ \M_2(t_k)  -  \M_2(0)}{t_k}  & = (\divv\VV)\varepsilon - D\VV\varepsilon - \varepsilon D\VV^\transp + \widetilde{D\varepsilon}\VV ,\\
 \label{eq:M4prime} \lim_{k \to \infty}\frac{ \M_4(t_k)  -  \M_4(0)}{t_k}  & = (\divv\VV)\f - D\VV\f + \widetilde{D\f}\VV.
\end{align}
Note that $\widetilde{\nabla\kappa}$ denotes the zero extension of $\nabla\kappa_{\vert B} \in   \boldsymbol{\mathcal C}(B)$ to $\Omega$. The same notation is used for  $\widetilde{D\E_d}, \widetilde{D\varepsilon}, \widetilde{D\nu},\widetilde{D\f}$.
Similarly, by the \Gateaux-differentiability of $\Lamg$ (see \cref{lemma:PropertiesReg}), \cref{eq:LamgBxig} and \cref{weakconP}, we deduce that    
\begin{multline}
 \label{eq:M3prime}
   \lim_{k \to \infty}\frac{ \M_3(t_k)  -  \M_3(0)}{t_k}  \cdot \P^\gamma_{t_k} 
=  \big((\divv\VV) \Lamg(\Egz) -D\VV\Lamg(\Egz)\big)\cdot\Pgz  \\- \bxig(\Egz)\Pgz\cdot(D\VV^\transp\Egz).
\end{multline}
From \cref{eq:Xiprime,eq:M0prime,eq:M1prime,eq:M2prime,eq:M3prime,eq:M4prime}  along with the weak convergence \cref{weakconP} and $\operatorname{supp} \VV \subset\subset B$,  it follows that
\begin{align}\label{eq:PartialDerivativeG}
&\lim_{k \to \infty}\frac{G(t_k,\Egz,\P^\gamma_{t_{k}})-G(0,\Egz,\P^\gamma_{t_{k}})}{t_k} \\\notag
 = & \,\int_{B}  \frac{1}{2}  (\nabla\kappa\cdot\VV + \kappa\divv\VV)|\Egz - \E_d|^2 - \kappa(\Egz - \E_d)\cdot (D\VV^\transp \Egz + D\E_d \VV)\,dx\\\notag
& + \int_{\omega} \divv\VV\,dx +\int_{B}\! \big(- (\divv\VV) \nu + D\VV^\transp \nu  +\nu D\VV + D\nu\VV\big) \curl\Egz\cdot  \curl\Pgz\,dx\\  \notag
& +\int_{B} \big((\divv\VV)\varepsilon - D\VV\varepsilon - \varepsilon D\VV^\transp + D\varepsilon\VV\big) \Egz\cdot \Pgz \,dx\\\notag
& +\int_{\om} (\divv\VV) \Lamg(\Egz)\cdot  \Pgz - D\VV \Lamg(\Egz)\cdot  \Pgz - \bxig(\Egz)\Pgz\cdot(D\VV^\transp\Egz)\,dx \\\notag
& -\int_B (D\f\VV + (\divv\VV )\f)\cdot \Pgz  - \f\cdot D\VV^\transp\Pgz\,dx \\\notag
=&\,\lim_{k \to \infty}\frac{G(t_k,\Egz,\Pgz)-G(0,\Egz,\Pgz)}{t_k} = \partial_t G(0,\Egz,\Pgz).
\end{align}
Thus, \ref{H4} is valid.
\end{proof}
It is easy to see that in the case $t=0$, the solution $\Pgz\in\Hzcurl$ of \cref{averated_adj} also satisfies the equation
\begin{align}\label{eq:LagrAdj} 
\partial_{\e}\ML(\om,\Egz,\Pgz;\hat\e) = 0 \quad\forall \hat\e\in  \Hzcurl.
\end{align}
By definition of the Lagrangian \cref{def:Lag} and by \cref{eq:M3Deriv} we conclude that \cref{eq:LagrAdj} is equivalent to
\begin{equation}\label{eq:AdjointEquation}
 a(\hat\e,\Pgz) + \int_\om \bxig(\Egz)\Pgz\cdot\hat\e \,dx = - \int_{B}  \kappa  (\Egz - \E_d)\cdot \hat\e\,dx, \quad\forall\hat\e\in\Hzcurl.
\end{equation}
We refer to \cref{eq:AdjointEquation} as the \textit{adjoint equation} and  we write for simplicity
$(\Eg,\Pg) =(\Egz,\Pgz)$.
We now  have all the elements at hand to prove the shape differentiability of $J_\gamma$  and write the distributed expression of the shape derivative of $J_\gamma$.
\begin{theorem}\label{thm:ShapeDerivative} \label{thm:sturm}
Let \cref{assump:MaterialAndData} be satisfied, $\gamma >0$,  $\om \in\mathcal{O}$ and $\VV \in \bm{\mathcal{C}}^{0,1}_c(\Omega)$ with a compact support in $B$. 
Furthermore, $\Eg\in\Hzcurl$ and $\Pg\in\Hzcurl$ denote the solutions to \cref{eq:RegVI,eq:AdjointEquation}, respectively. 
Then, the functional $J_\gamma$ in \cref{eq:Preg} is shape differentiable with
\begin{equation}\label{eq:ShapeDerivative}
dJ_\gamma(\om)(\VV) = \partial_t G(0,\Eg, \Pg)= \int_{B}  S_1^\gamma : D\VV + \bS_0^\gamma\cdot\VV\,dx,
\end{equation}
where $S_1^\gamma\in L^1(B,\R^{3\times 3})$ and $\bS_0^\gamma\in {\bm L}^1(B)$  are given by
\begin{align*}
S_1^\gamma & = \Big[\frac{\kappa}{2} |\Eg - \E_d|^2 + \chi_\om
- \nu \curl\Eg\cdot  \curl\Pg
+ \varepsilon \Eg\cdot \Pg 
+ \chi_\om \Lamg(\Eg)\cdot\Pg\\
&\qquad-\f\cdot\Pg\Big]  \bI_3
%
- \kappa \Eg\otimes (\Eg - \E_d)
+\nu \curl\Eg\otimes\curl\Pg \\
&\quad+ \nu^\transp \curl\Pg\otimes\curl\Eg 
%
- \Pg \otimes \varepsilon \Eg - \Eg\otimes \varepsilon^\transp \Pg + \Pg\otimes\f \\
%
&\quad- \chi_\om\Lamg(\Eg)\otimes \Pg 
- \Eg \otimes \bxig(\Eg)\Pg ,\\
\bS_0^\gamma & = \frac{\nabla\kappa}{2} |\Eg - \E_d|^2 - \kappa D\E_d^\transp (\Eg -\E_d)
+ (D\nu^\transp\curl\Eg)  \curl\Pg \\
&\qquad+ (D\epsilon^\transp\Eg)\Pg - D\f^\transp \Pg.
\end{align*}
\end{theorem}
\begin{proof}
 Thanks to \cref{lemma:H0,lemma:H1}, we may apply the averaged adjoint method (see \cref{thm:AAE}). This yields that $J_\gamma$ is shape-differentiable in the sense of \cref{def1} and the shape derivative satisfies
 \begin{equation}\label{eq:RepresShapeDeriv}
 dJ_\gamma(\om)(\VV) =  \frac{d}{dt} J_\gamma(\om_t)|_{t=0} = \partial_t G(0,\Eg,\Pg),
 \end{equation}
 where $\partial_t G(0,\Eg,\Pg)$ is given by \cref{eq:PartialDerivativeG}. 
We note that $D\epsilon, D\nu \colon \Omega \to \R^{3\times 3\times 3}$ are third-order tensors, and their transpose $D\epsilon^\transp$, $D\nu^\transp$ satisfy $(D\epsilon\VV)\Eg\cdot\Pg = (D\epsilon^\transp\Eg)\Pg\cdot\VV$, and  $(D\nu \VV)\curl\Eg  \cdot\curl\Pg  = (D\nu^\transp\curl\Eg)  \curl\Pg \cdot\VV$ ; see \cite[Proposition 3.1]{qi2017transposes}.  
Furthermore, for vectors ${\bm x},{\bm y}\in\R^3$ we have the relations
$D\VV : ({\bm x}\otimes {\bm y}) = {\bm x}\cdot D\VV {\bm y} = D\VV^\transp {\bm x}\cdot  {\bm y}.$
Applying  these 
to \cref{eq:PartialDerivativeG} and combining it with \cref{eq:RepresShapeDeriv}, the tensor expression \cref{eq:ShapeDerivative} for the shape derivative follows.
Finally, the fact that $S_1^\gamma\in L^1(B,\R^{3\times 3})$ and $\bS_0^\gamma\in {\bm L}^1(B)$ is a straightforward consequence of the regularity of $\Eg, \Pg$ and of the  other functions involved in the expressions of $\bS_0^\gamma$ and $S_1^\gamma$.
This completes the proof.
\end{proof}

\section{Stability and convergence analysis} \label{section:StabilityAnalysis}
In this section we analyze the stability of the shape derivative \cref{eq:ShapeDerivative} with respect to the penalization parameter $\gamma>0$.
Furthermore, the strong convergence of   \cref{eq:Preg} towards   \cref{eq:P} as $\gamma\to\infty$ is studied. The latter also implies the existence of an optimal shape for \cref{eq:P} (see \cref{thm:ExistenceP}). 

\subsection{Stability analysis of the shape derivative}
\begin{theorem}\label{thm:ConvergenceShapeDerivative}
 Let $\om\in\mathcal O$ and \cref{assump:MaterialAndData} hold. Then, the following stability estimate holds
\begin{equation}\label{eq:StabilityShapeDeriv}
 |dJ_\gamma(\om)(\VV)| \leq C \norm{\VV}{\bm{\mathcal{C}}^{0,1}(B)} \quad \forall \VV\in \bm{\mathcal{C}}^{0,1}_c(\Omega), \, \operatorname {supp} \VV \subset\subset B, 
\end{equation}
with a constant $C = C(j_c, \kappa, \epsilon, \nu, \f,  \E_d, B, \omega)$ independent of $\gamma$. 
\end{theorem}
\begin{proof}
First of all, the distributed shape derivative from \cref{eq:ShapeDerivative} yields the estimate
\begin{equation}\label{eq:StabilityShapeDeriv0}
 |dJ_\gamma(\om)(\VV)| \leq \big(\|S_1^\gamma\|_{L^1(B, \R^{3\times3})} + \|\bS_0^\gamma\|_{{\bm L}^1(B)}\big) \norm{\VV}{\bm{\mathcal{C}}^{0,1}(B)}.
\end{equation}
In order to derive upper bounds for $\|S_1^\gamma\|_{L^1(B, \R^{3\times3})}$ and $ \|\bS_0^\gamma\|_{{\bm L}^1(B)}$, we begin by proving that the families $\{\Eg\}_{\gamma>0}$ and $\{\Pg\}_{\gamma>0}$ 
are uniformly bounded in $\Hzcurl$. In view of \cref{eq:BoundednessEgt}, we have
\begin{equation}\label{eq:BoundEG}
 \Hcurlnorm{\Eg} \leq \min(\underline{\nu},\underline{\epsilon})^{-1} (\Lnorm{\f} + j_c) = C_\E.
\end{equation}
Moreover, we set $t, s = 0$ in \cref{eq:M3Deriv}, which yields
\begin{equation}\label{eq:deM3pos}
\int_\om \partial_\e \M_3(0,\ME(0))(\Pg)\cdot\Pg \,dx = \int_\om \bxig(\Eg)\Pg\cdot\Pg \,dx \geq 0.
\end{equation}
In fact, the non-negativity of \cref{eq:deM3pos} follows by similar calculations as \cref{eq:ToEstimate,eq:EstPositivePart,eq:EstAgammaPart,eq:EstSgammaPart,eq:CoercM3} in the special case $t,s,\tau =0$.
As $\Pg$ is the unique solution to \cref{eq:AdjointEquation}, inserting $\hat\e = \Pg$ implies with \ref{assump:Material}
\begin{multline*}
 \min(\underline{\epsilon},\underline{\nu}) \Hcurlnorm{\Pg}^2 \leq a(\Pg,\Pg) \\
 = -\int_{B} \kappa(\Eg - \E_d)\cdot \Pg\,dx - \int_\om \bxig(\Eg)\Pg\cdot\Pg \,dx.
 \end{multline*}
Hence, we obtain a uniform bound for $\Pg$ by means of
\cref{eq:BoundEG,eq:deM3pos}, i.e.,
\begin{equation}\label{eq:BoundPG}
 \quad \Hcurlnorm{\Pg} \leq  \norm{\kappa}{\mathcal{C}(\Omega)}\min(\underline{\epsilon},\underline{\nu})^{-1} \big(C_\E + \LnormB{\E_d}\big) = C_\P.
\end{equation}
With \cref{eq:BoundEG,eq:BoundPG} we may now estimate both terms in \cref{eq:StabilityShapeDeriv0} separately. 
Therefore, let us introduce the notation (see \cref{thm:ShapeDerivative})
\begin{equation}
 S^\gamma_1 \coloneqq \sum_{i=1}^{14} \Theta_i.
\end{equation}
where $\Theta_i \in L^1(B, \R^{3\times3})$ for every $i \in \{1,\dots,14\}$. Now, \cref{assump:MaterialAndData}, \cref{eq:BoundEG,eq:BoundPG} together with H{\" o}lder's and Young's inequalities yield
\begin{align}\label{eq:Theta1-6}
 &\sum_{i=1}^6\norm{ \Theta_i}{L^1(B,\R^{3\times3})} \\\notag
\leq\quad\;\; &\, \int_{B}  \frac{|\kappa|}{2}|\Eg - \E_d|^2 + \chi_\om \, dx + \int_B |\nu\curl\Eg\cdot\curl\Pg| + |\epsilon\Eg\cdot\Pg|\, dx \\\notag  
    &  + \int_\omega |\Lamg(\Eg)\cdot\Pg|\,dx  + \int_B|\f\cdot\Pg| \, dx\\\notag
 \overbrace{\leq}^{\cref{eq:BoundEG}\&\cref{eq:BoundPG}} &\, \norm{\kappa}{\mathcal{C}(B)}\big(C_\E^2 + \LnormB{\E_d}^2\big) + |\omega| + \big(\norm{\nu}{\mathcal C(B,\R^{3\times3})} + \norm{\epsilon}{\mathcal C(B,\R^{3\times3})}\big)C_\E C_\P \\\notag 
    &+ (j_c\sqrt{|\omega|} + \LnormB{\f})C_\P
\end{align}
For the remaining terms, we   use again \cref{assump:MaterialAndData}, \cref{eq:BoundEG,eq:BoundPG} as well as the identity $| \bm x \otimes \bm y| = |\bm x| \cdot|\bm y|$ for all  $\bm x,\bm y \in \R^3$ to infer 
\begin{multline}\label{eq:Theta7-13}
 \sum_{i=7}^{13} \norm{ \Theta_i}{L^1(B,\R^{3\times3})}  \leq \frac{1}{2}\norm{\kappa}{\mathcal C(B)}\big(3C_\E^2 + \LnormB{\E_d}^2\big) \\
    + 2(\norm{\nu}{\mathcal C(B,\R^{3\times3})} + \norm{\epsilon}{\mathcal C(B,\R^{3\times3})})C_\E C_\P + (\norm{\f}{\bm L^2(B)} + j_c \sqrt{|\omega|}) C_\P,
\end{multline}
where we have also used Young's inequality to obtain the first term in \cref{eq:Theta7-13}.
Moreover, we may estimate the last summand of $S_1^\gamma$ as follows
\begin{align}\label{eq:Theta14}
 &\norm{\Theta_{14}}{L^1(\Omega, \R^{3\times3})} = \norm{\Eg\otimes\bxig(\Eg)\Pg}{L^1(\Omega, \R^{3\times3})} \leq \int_\omega |\bxig(\Eg)\Pg|\cdot |\Eg|\,dx \\\notag
 \overbrace{\leq}^{\cref{eq:SgammaEst}\&\cref{eq:DefiBxig} } &\int_\om \left(\frac{j_c\gamma|\Pg|}{\maxg\{1, \gamma|\Eg|\}} + \frac{\gamma |\Eg \otimes \Lamg(\Eg)| \cdot |\Pg|}{\maxg\{1,\gamma|\Eg|\}|\Eg|} \right)|\Eg|\,dx \\\notag
 \overbrace{\leq}^{\cref{eq:LamgJc}} \quad &\int_\om 2j_c|\Pg|\,dx \leq 2j_c \sqrt{|\om|}C_\P.
\end{align}
Gathering \cref{eq:Theta1-6,eq:Theta7-13,eq:Theta14}  we deduce the final estimate for $S^\gamma_1$
\begin{multline}\label{eq:S1Est}
 \norm{S^\gamma_1}{L^1(B,\R^{3\times3})} \leq \frac{1}{2}\norm{\kappa}{\mathcal C(B)}\big(5 C_\E^2 + 3\LnormB{\E_d}^2\big)  + |\omega|  \\
 +3\big( \norm{\nu}{\mathcal C(B,\R^{3\times3})} + \norm{\epsilon}{\mathcal C(B,\R^{3\times3})} \big) C_\E C_\P + \big(2\LnormB{\f} + 4j_c\sqrt{|\om|}\big) C_\P.
\end{multline}
Again, \cref{eq:BoundEG,eq:BoundPG} with H{\"o}lder's and Young's inequalities imply for $\bS^\gamma_0$
\begin{align}\notag
 \norm{\bS^\gamma_0}{\bm L^1(B)} \leq &\, \int_B \frac{1}{2} |\nabla \kappa| \cdot|\Eg - \E_d|^2 + |\kappa D\E_d^\transp (\Eg - \E_d)| \,dx \\\notag
    &+ \int_B |D\nu^\transp\curl\Eg|\cdot |\curl\Pg| + |D\epsilon^\transp\Eg|\cdot|\Pg| + |D\f^\transp \Pg|\,dx\\\notag
 \leq &\,\frac{1}{2}\norm{\kappa}{\mathcal C^1(B)}\big(3C_\E^2 + 5\norm{\E_d}{\bm H^1(B)}^2\big) \\ \label{eq:S2Est}
    &+ \big(\ConematnormB{\nu} + \ConematnormB{\epsilon}\big)C_\P C_\E + \norm{\f}{\bm H^1(B)} C_\P.
\end{align}
Finally, we combine \cref{eq:StabilityShapeDeriv0,eq:S1Est,eq:S2Est} to conclude
\begin{multline*}
 |dJ_\gamma(\om)(\VV)| \leq \!\bigg[4\norm{\kappa}{\mathcal C^1(B)}\big(C_\E^2 + \norm{\E_d}{\bm H^1(B)}^2\big)+ 4\big(\ConematnormB{\nu} + \ConematnormB{\epsilon}\big)C_\E C_\P \\ 
   +|\omega|  + \big(3 \norm{\f}{\bm H^1(B)} + 4j_c\sqrt{|\om|}\big)C_\P\bigg] \norm{\VV}{\bm{\mathcal{C}}^{0,1}(B)}
\end{multline*}
Hence, the proof is finished.

\end{proof}

\subsection{Convergence of the regularized shape optimization problem}
Our aim is to prove the strong convergence of \cref{eq:Preg} towards \cref{eq:P}. 
For this purpose, we recall a helpful result which states the strong convergence of the solution to   \cref{eq:RegVI} for a fixed $\om\in\mathcal{O}$.
A proof can be found in \cite[Corollary 4.3]{DeLosReyes11}:
\begin{lemma}\label{lemma:ConvFixedOmega}
 Let \cref{assump:MaterialAndData} be satisfied and $\omega \in \mathcal{O}$. Moreover, for every $\gamma >0$, let $(\Eg, \lamg) \in \Hzcurl \times  {\bm L}^\infty(\om)$ denote the solution to \cref{eq:RegVI}.  Then,  \begin{equation}\label{eq:StrongConvEgFixedOmega}
 \begin{aligned}
  (\Eg, \lamg) &\to (\E,\lam) &&\text{ strongly in } \Hzcurl \times \Hzcurl^* \text{ as } \gamma \to \infty.
 \end{aligned}
 \end{equation}
where $(\E, \lam) \in \Hzcurl \times  {\bm L}^\infty(\om)$ is the unique solution to    \cref{eq:LagrangeVI}. 
\end{lemma}
Let us point out  that in \cref{eq:StrongConvEgFixedOmega} we extended the Lagrange multipliers  $\lamg,\lam$ by zero as functions in $\L$, i.e., we set  $\lamg(x) = 0$ and $\lam(x)=0$ for all $x\in \O\backslash \om$. This zero extension shall  also be   used in the following theorem.
\begin{theorem}\label{thm:ConvergenceEgLamg}
Let \cref{assump:MaterialAndData} hold and $\{\gamma_n\}_{n\in\N}\subset \R^+$ be such that 
$ \gamma_n \to \infty$ as $n\to\infty$.
Then, there exists a subsequence of $\{\gamma_n\}_{n\in\N}$, still denoted by $\{\gamma_n\}_{n\in\N}$, such that the sequence of solutions $\{\om^{\gamma_n}\}_{n\in\N}$ of \cref{eq:Preg} with $\gamma = \gamma_n$ converges towards an optimal solution $\oms\subset\mathcal{O}$ of \cref{eq:P} in the sense of Hausdorff and in the sense of characteristic functions. 

Moreover, $\{(\E^{\gamma_n}(\om^{\gamma_n}), \lam^{\gamma_n}(\om^{\gamma_n}))\}_{n \in\N}$ and $(\E(\omt),\lam(\omt))$ as the solutions of \cref{eq:RegVI} for $\om=\om^{\gamma_n}$ and \cref{eq:LagrangeVI} for $\om=\omt$, respectively, satisfy
 \begin{align}
\label{eq:StrongConvEg}  &\lim_{\gamma \to \infty} \Hcurlnorm{\E^{\gamma_n}(\om^{\gamma_n}) - \E(\omt)} = 0, \\ 
\label{eq:StrongConvLamg}    &\lim_{\gamma \to \infty} \norm{\lam^{\gamma_n}(\om^{\gamma_n}) - \lam(\omt)}{\Hzcurl^*} = 0,
 \end{align}
where $\lam^{\gamma_n}(\om^{\gamma_n})$ \emph{(}resp. $ \lam(\omt)$ \emph{)} is extended by zero in $\O \setminus \om^{\gamma_n}$ \emph{(}resp. in  $\O \setminus \omt$\emph{)}.
\end{theorem}
\begin{proof}
Thanks to \cref{thm:lipCV} and $\gamma_n \to \infty$, there exists $\oms \in \mathcal{O}$ such that, possibly for a subsequence,
 \begin{equation}\label{eq:ConvOmg}
  \om^{\gamma_n} \to \omt \quad \text{as } n \to \infty
 \end{equation}
 in the sense of Hausdorff and in the sense of characteristic functions.
Furthermore, we have the estimate 
\begin{multline}\label{eq:SplitConv}
 \Hcurlnorm{\E^{\gamma_n}(\om^{\gamma_n}) - \E(\omt)} \leq \Hcurlnorm{\E^{\gamma_n}(\om^{\gamma_n}) - \E^{\gamma_n}(\omt)}\\
 + \Hcurlnorm{\E^{\gamma_n}(\omt) - \E(\omt)}.
\end{multline}
Now, by  virtue of \cref{lemma:ConvFixedOmega}, the second term on the right-hand side of \cref{eq:SplitConv} converges to $0$ as $n \to \infty$. For the first term we observe (for every $n \in\N$) that the arguments used to derive \cref{eq:EstimateEgammaEngamma} are applicable. Thus, we substract \cref{eq:RegVI} for $\E^{\gamma_n}(\om^{\gamma_n})$ and \cref{eq:RegVI} for $\E^{\gamma_n}(\omt)$ and test the resulting equation with $\vb = \E^{\gamma_n}(\omt) - \E^{\gamma_n}(\om^{\gamma_n})$. Hereafter, analoguously to  \cref{eq:CalcConvergenceEngamma}, calculations involving \cref{eq:MonotoneLam} yield
\begin{equation}\label{eq:EstimateEgammaEngamma2}
 \Hcurlnorm{\E^{\gamma_n}(\om^{\gamma_n}) - \E^{\gamma_n}(\omt)} \leq \frac{j_c}{\min\{\underline{\nu}, \underline{\epsilon}\}} \norm{\chi_\omt - \chi_{\om^{\gamma_n}}}{L^2(\Omega)} \quad\forall n\in\N.
\end{equation}
Combining \cref{lemma:ConvFixedOmega} and \cref{eq:ConvOmg,eq:SplitConv,eq:EstimateEgammaEngamma2} together leads to   \cref{eq:StrongConvEg}.

Furthermore, substracting \cref{eq:LagrangeVI} for $\om = \oms$ and \cref{eq:RegVI} for $\om = \om^{\gamma_n}$ implies
\begin{align}
 &\sup_{\vb\in\Hzcurl} \frac{(\lam^{\gamma_n}(\om^{\gamma_n}) - \lam(\omt), \vb)_\L}{\Hcurlnorm{\vb}} =  \sup_{\vb\in\Hzcurl} \frac{a(\E(\omt) - \E^{\gamma_n}(\om^{\gamma_n}), \vb)}{\Hcurlnorm{\vb}} \\\notag
 \overbrace{\leq}^{\ref{assump:Material}}\, &\max\{ \norm{\epsilon}{{L}^\infty(\Omega,\R^{3\times3})}, \norm{\nu}{{L}^\infty(\Omega,\R^{3\times3})}\} \Hcurlnorm{\E(\omt) - \E^{\gamma_n}(\om^{\gamma_n})}.
\end{align}
Thus, \cref{eq:StrongConvLamg} follows from \cref{eq:StrongConvEg}. 
It remains to verify that $\omt \in \mathcal{O}$ is in fact a minimizer of \cref{eq:P}. 
First of all, we note that, since $\om^{\gamma_n}$ is a solution of \cref{eq:Preg} for $\gamma = \gamma_n$, the following estimate holds
\begin{equation}\label{eq:EstimateMinimizer}
J_{\gamma_n}(\om^{\gamma_n}) = \min_{\om \in \mathcal{O}} J_{\gamma_n} (\om) \leq J_{\gamma_n} (\om) \quad \forall \om \in \mathcal{O}.
\end{equation}
Finally, gathering all the previous results, we obtain for every $\om\in\mathcal{O}$ that
\begin{align*}
J(\omt) =\qquad & \frac{1}{2} \int_B \kappa |\E(\omt) - \E_d|^2 dx + \int_\omt dx \\
\overbrace{=}^{\cref{eq:StrongConvEg}\&\cref{eq:ConvOmg}}& \lim_{n \to \infty} \frac{1}{2} \int_B \kappa |\E^{\gamma_n}(\om^{\gamma_n}) - \E_d|^2 dx + \int_{\om^{\gamma_n}} dx = \lim_{n \to \infty} J_{\gamma_n}(\om^{\gamma_n})\\
 {\overbrace{\leq}^{\cref{eq:EstimateMinimizer}}}\quad\,\,&   \lim_{n \to \infty } J_{\gamma_n}(\om) = \lim_{n \to \infty} \frac{1}{2} \int_B \kappa | \E^{\gamma_n}(\om) - \E_d|^2 dx + \int_\om dx \\
 {\overbrace{=}^{\cref{eq:StrongConvEgFixedOmega}}}\quad\,\,&\;  \frac{1}{2} \int_B \kappa |\E(\om) - \E_d|^2 dx + \int_\om dx = J(\om).
\end{align*}
This shows $J(\omt) \leq J(\om)$ for every $\om \in \mathcal{O}$ which yields the assertion.
\end{proof} 
\begin{remark}\label{rem:ProofExistenceP}
 As we have obtained the optimal shape $\oms\in\mathcal{O}$ in \cref{eq:ConvOmg} as the limit of the optimal shapes for \cref{eq:Preg}, \cref{thm:ExistenceP} follows immediately from \cref{thm:ConvergenceEgLamg}.
\end{remark}

\section{Numerical tests} \label{section:NumericalTests}
Our algorithm to obtain a numerical approximation for the optimal shape $\omt$ of \cref{eq:P} is based on a variant of the level set method where the distributed shape derivative (\cref{thm:ShapeDerivative}) is used to obtain a descent direction (see \cite{MR3535238}). 
We refer to \cite{Laurain2018} for a detailed description of this algorithm including its implementation in a 2D framework. 
We consider the proposed approach \cref{eq:Preg} with $\gamma = 7\cdot 10^4$. 
The forward problems \cref{eq:RegVI} are computed using the  Newton method 
with a finite element discretization based on the first family of \Nedelec's edge elements \cite{ned80} at roughly 2.000.000 DoFs. As announced in the introduction, we apply our algorithm to two problems stemming from high-temperature superconductivity (HTS), also widely known as type-II superconductivity.

We choose $\Omega = [-2,3]^3$ and $B = [0,1]^3$. 
For simplicity, we take the material parameters $\epsilon = \nu = \bI_3$ (cf. \ref{assump:Material}). 
Moreover, $\f$ is a circular current
\begin{align*}
\f(x,y,z) = \left\{\begin{aligned}
&\frac{R}{\sqrt{(y - 0.5)^2+(z - 0.5)^2}}\left(0 ,\; -z + 0.5 ,\; y - 0.5\right)&&\text{ for } (x,y,z)\in\Omega_p,\\
&0 &&\text{ for } (x,y,z) \notin \Omega_p,
\end{aligned}\right.
\end{align*} 
applied to a pipe coil $\Omega_p\subset \Omega$ which is defined by
\begin{equation*}
 \Omega_p \coloneqq \left\{ (x,y,z)\in\Omega \,:\, |z - 0.5| \leq 0.5, \, \sqrt{(x - 0.5)^2 + (y - 0.5)^2} \in [1.2,1.6] \right\}.
\end{equation*}
The constant $R>0$ denotes the electrical resistance of $\Omega_p$ (here: $R = 10^{-3}$).  As $\Omega_p \cap B = \emptyset$, we have $\f \equiv 0$ in $B$ and \ref{assump:f} is satisfied. Without a superconductor in the system, this current would induce an orthogonal magnetic field which admits its highest field strength inside the coil.

We use the distributed expression \cref{eq:ShapeDerivative} of the shape derivative to obtain a descent direction ${\bm\Theta}$. More precisely,  let $\V_h \subset \H^1(B) \cap \bm{\mathcal{C}}^{0,1}(\overline B)$ be the space of piecewise linear and continuous finite elements on $B$.
Given a positive definite bilinear form $\mathcal{B}: \V_h \times \V_h \to\R$, the problem is to find ${\bm\Theta}\in  \V_h$ such that    
\begin{equation}
\label{VP_1}
\mathcal{B}({\bm\Theta},{\bm\xi}) =-  dJ_\gamma(\om)({\bm\xi}) \mbox{  for all  } {\bm\xi}\in \V_h.
\end{equation}
With this choice, the solution ${\bm\Theta}$ of \cref{VP_1} is defined on  $B$ and is a descent direction since $dJ_\gamma(\om)({\bm\Theta}) = -\mathcal{B}({\bm\Theta},{\bm\Theta})< 0$ if ${\bm\Theta}\neq 0$. 
In our algorithm we choose 
\begin{equation}
\label{VP_2}
\mathcal{B}({\bm\Theta},{\bm\xi}) = \int_B \alpha_1 D{\bm\Theta} : D{\bm\xi} +  \alpha_2 {\bm\Theta}\cdot {\bm\xi}\,dx +\alpha_3 \int_{\partial B} ({\bm\Theta} \cdot \n) ({\bm\xi} \cdot \n) \,ds,
\end{equation}
with $\alpha_1 = 0.5$, $\alpha_2 = 0.5$ and $\alpha_3=1.0$. 
Moreover, the geometry was optimized in the class of shapes with two symmetries with respect to the planes $x = 0.5$ and $y = 0.5$.
This is achieved by symmetrizing ${\bm\Theta}$  with respect to these axis, and it can be shown that the symmetrized vector field is still a descent direction according to the symmetrization  technique proposed in Section \ref{sec:sym}.

All codes are written in {\scshape Python} with the open-source finite-element computational software {\scshape FEniCS} \cite{fenics:book}. 
We used {\scshape Paraview} to visualize the 3D plots.


 \begin{figure}[tbhp]
  \vspace{-0.1cm}
  \centering
\subfloat{\label{fig:Ex1:Shape_0}\includegraphics[width=0.24\textwidth]{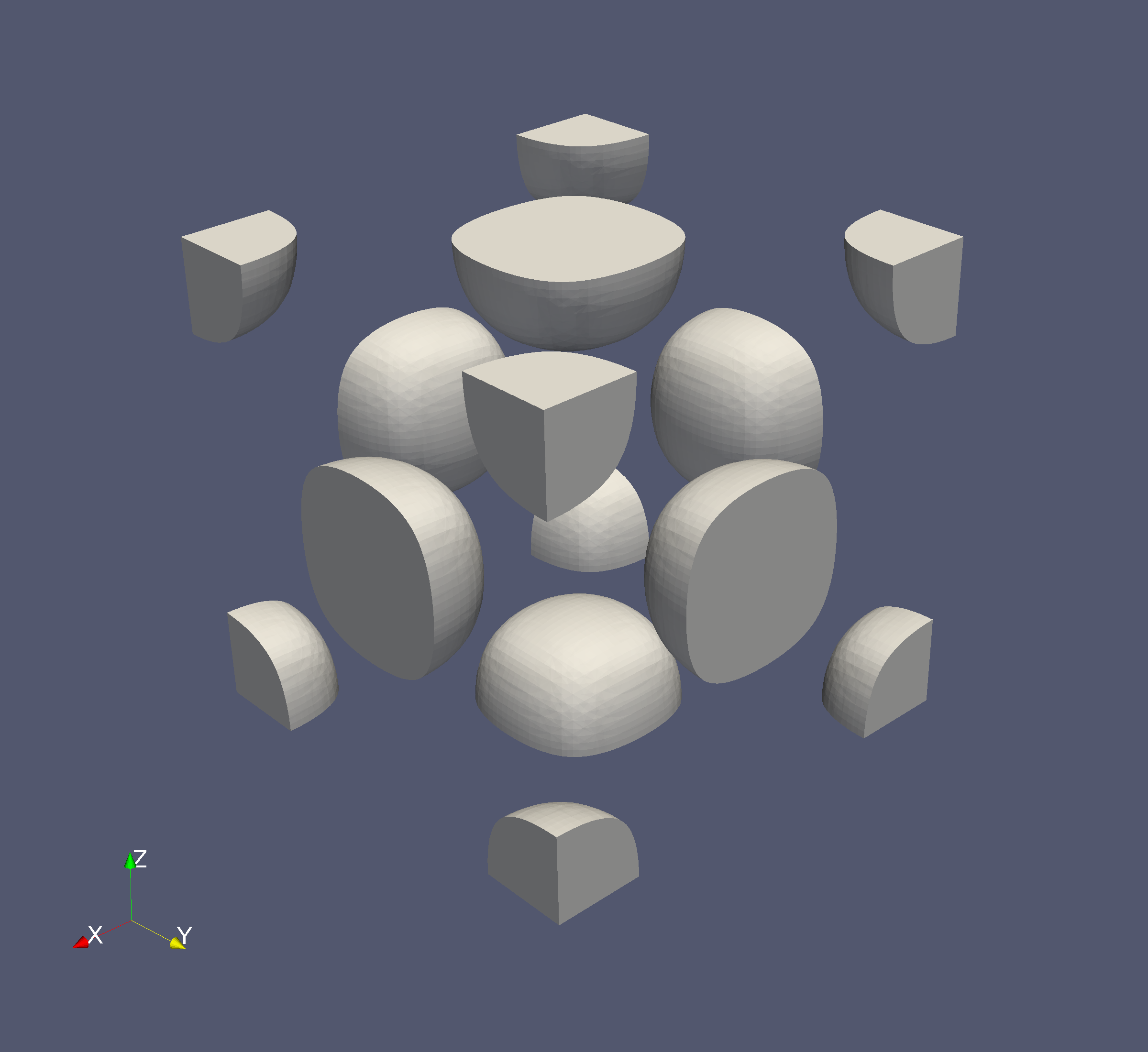}}\,
\subfloat{\label{fig:Ex1:Shape_1}\includegraphics[width=0.24\textwidth]{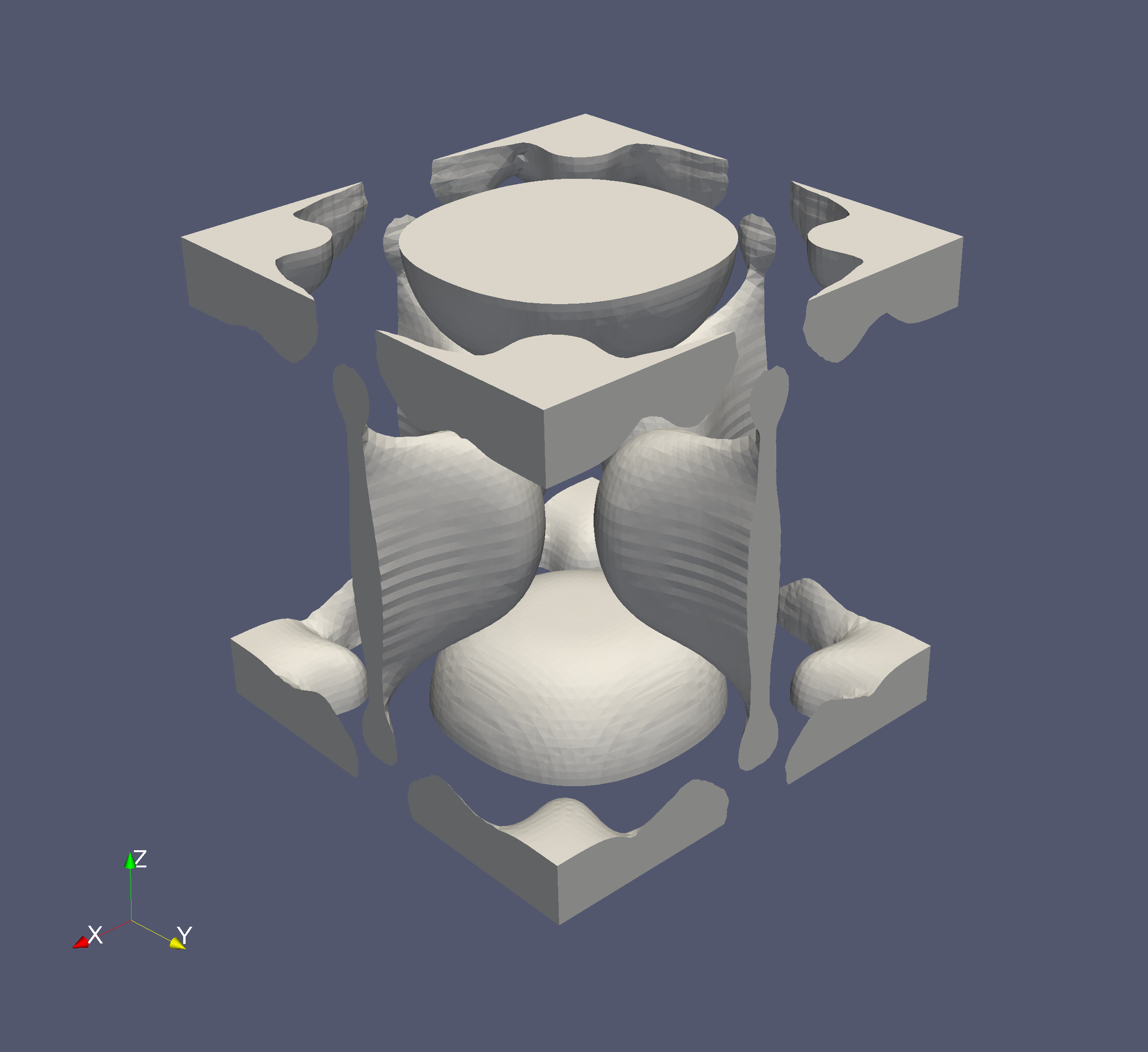}}\,
\subfloat{\label{fig:Ex1:Shape_2}\includegraphics[width=0.24\textwidth]{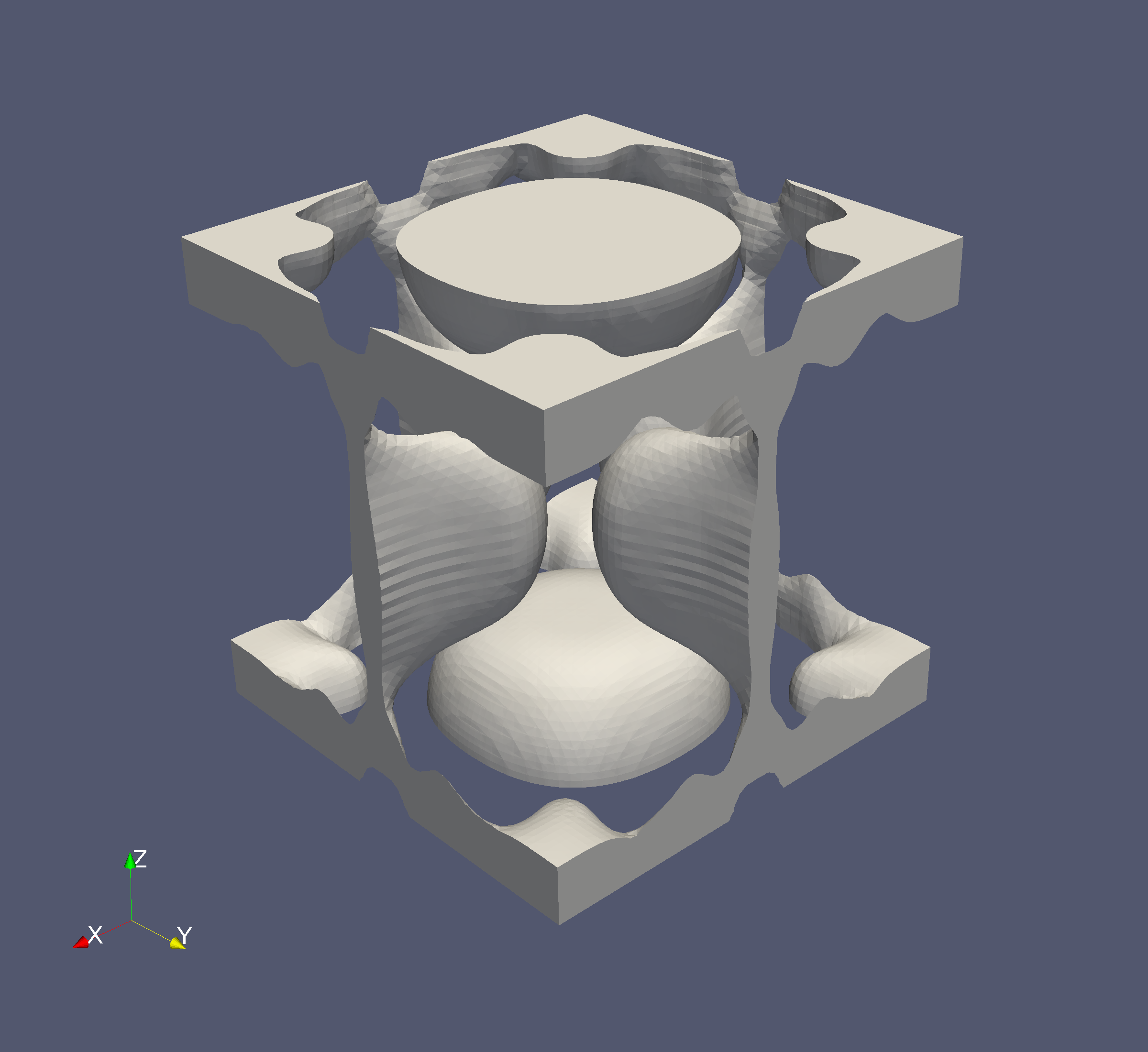}}\, 
\subfloat{\label{fig:Ex1:Shape_3}\includegraphics[width=0.24\textwidth]{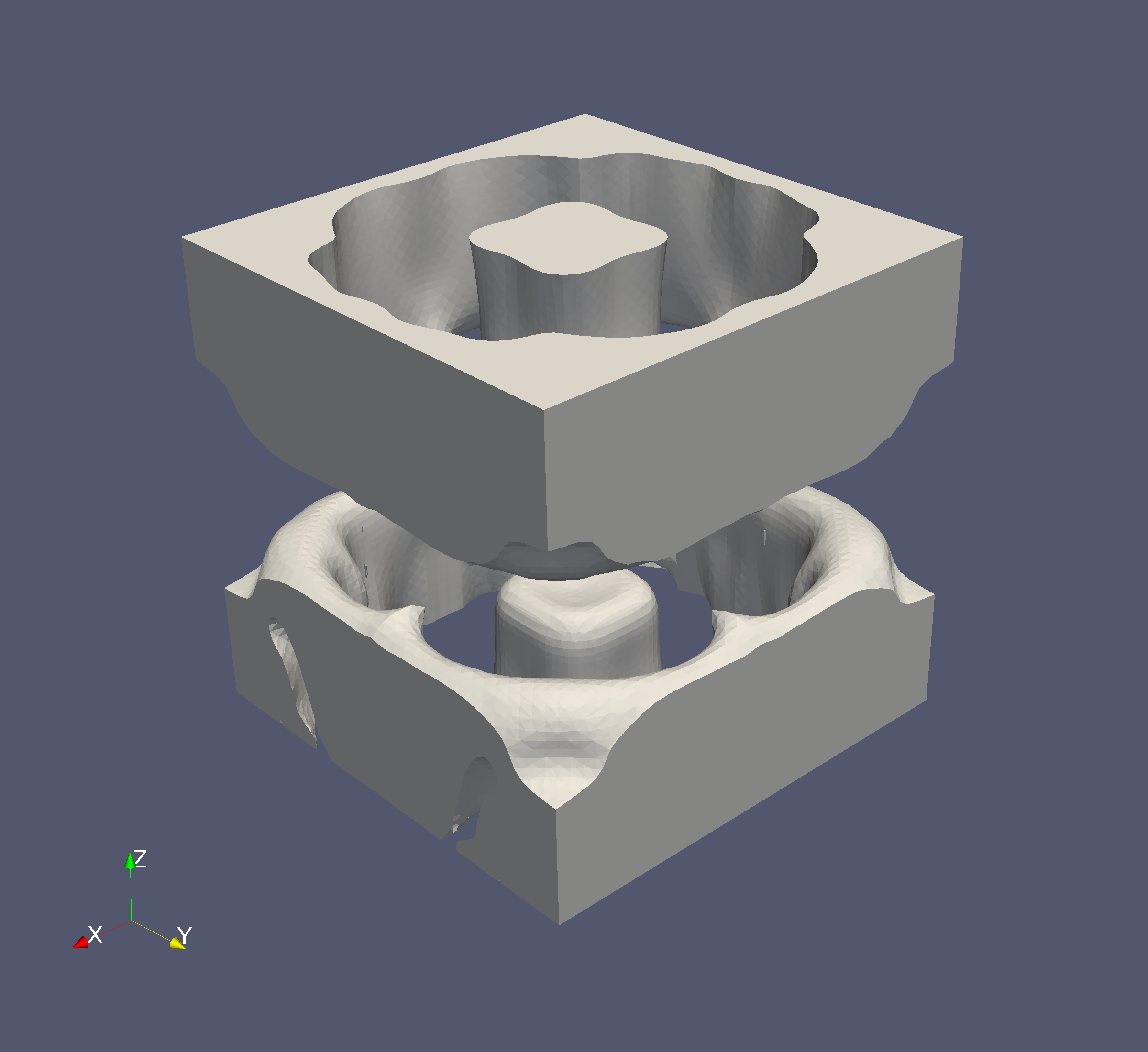}} 
\caption{Shapes generated by the algorithm at iterations $0,42,45,143$.}
\label{fig:Ex1:Shapes} 
\centering
\subfloat{\label{fig:Ex1:Init_Field_1}\includegraphics[width=0.24\textwidth]{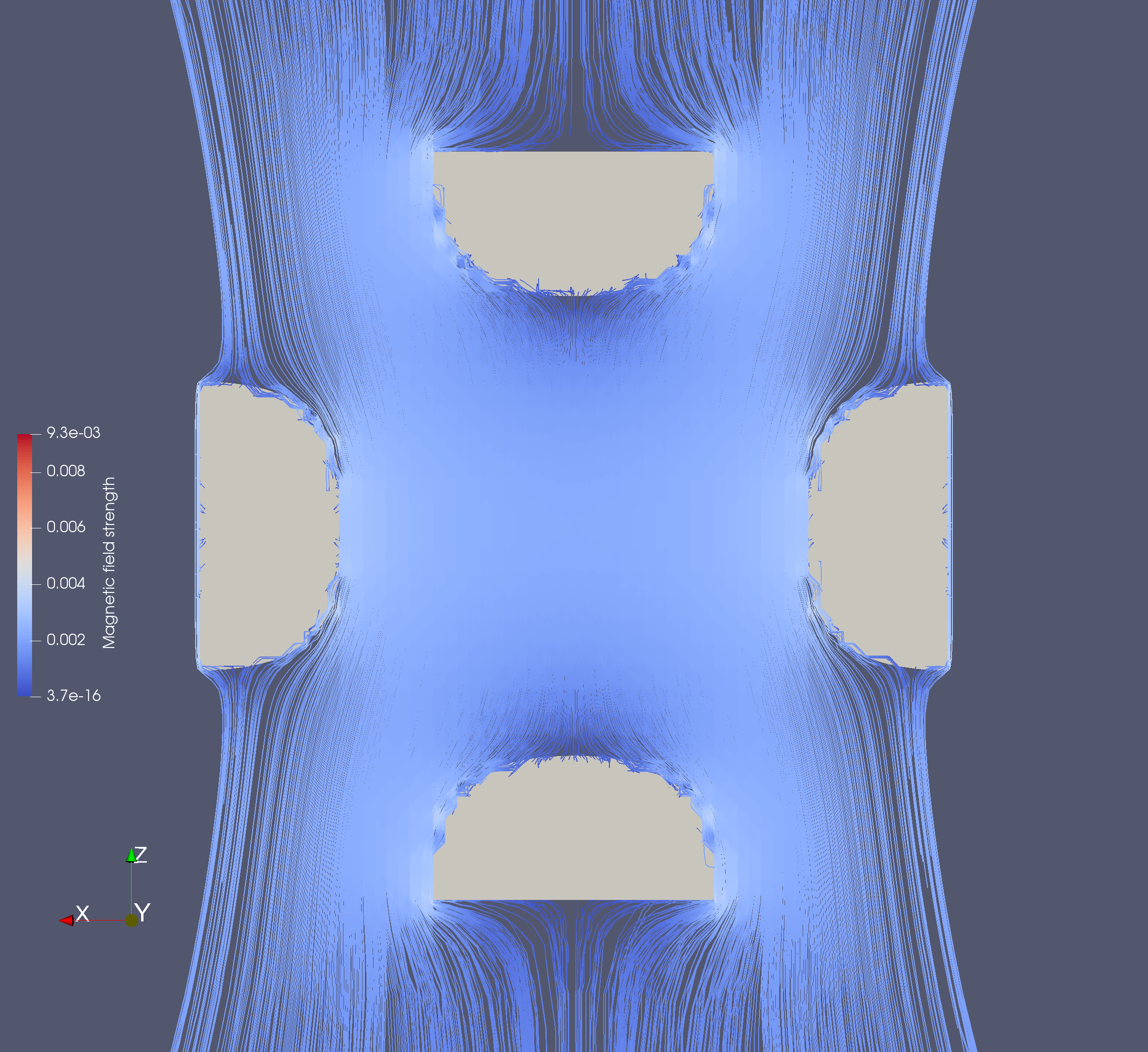}}\, 
\subfloat{\label{fig:Ex1:Field_1}\includegraphics[width=0.24\textwidth]{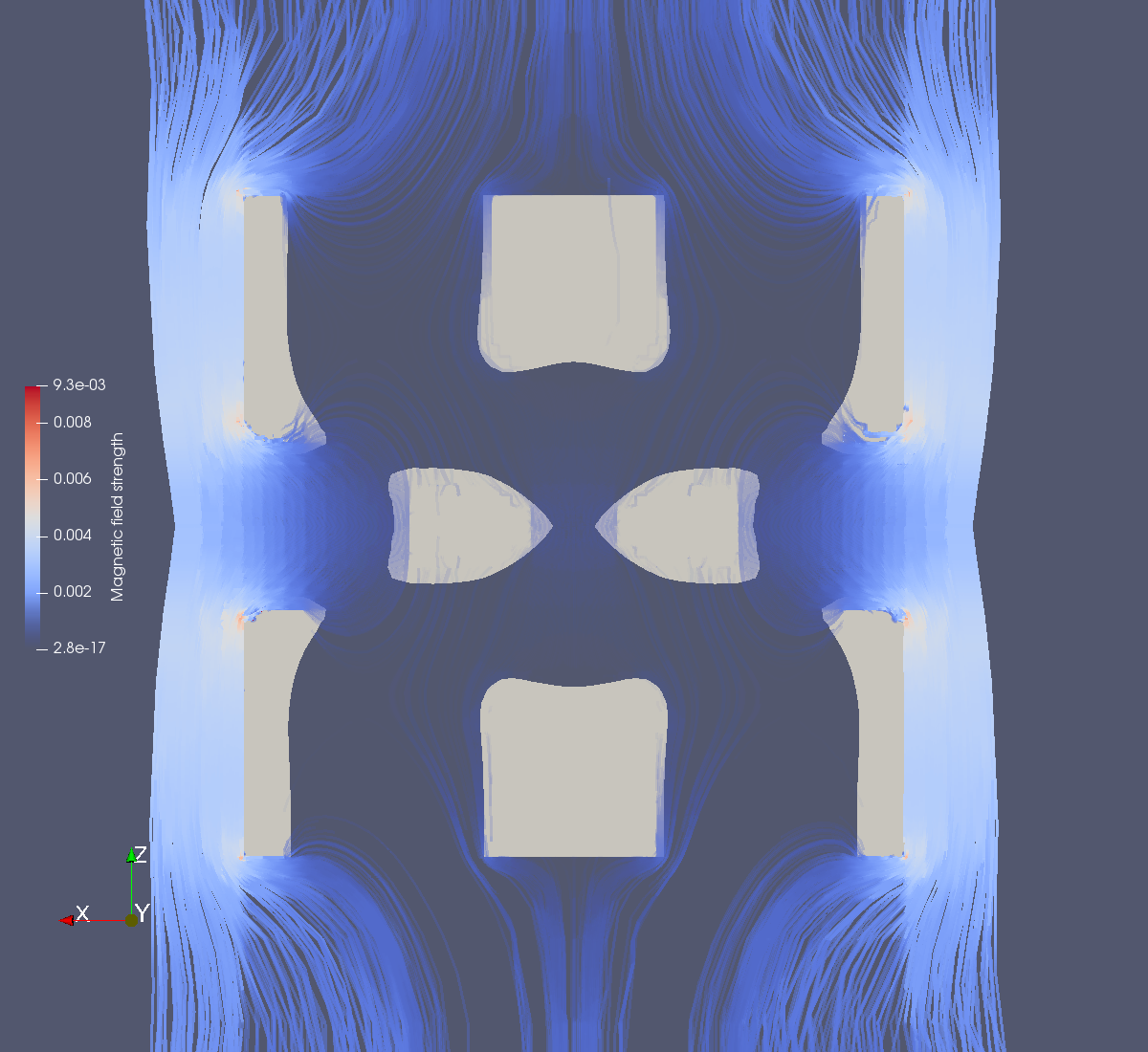}}\,
\subfloat{\label{fig:Ex1:Init_Field_2}\includegraphics[width=0.24\textwidth]{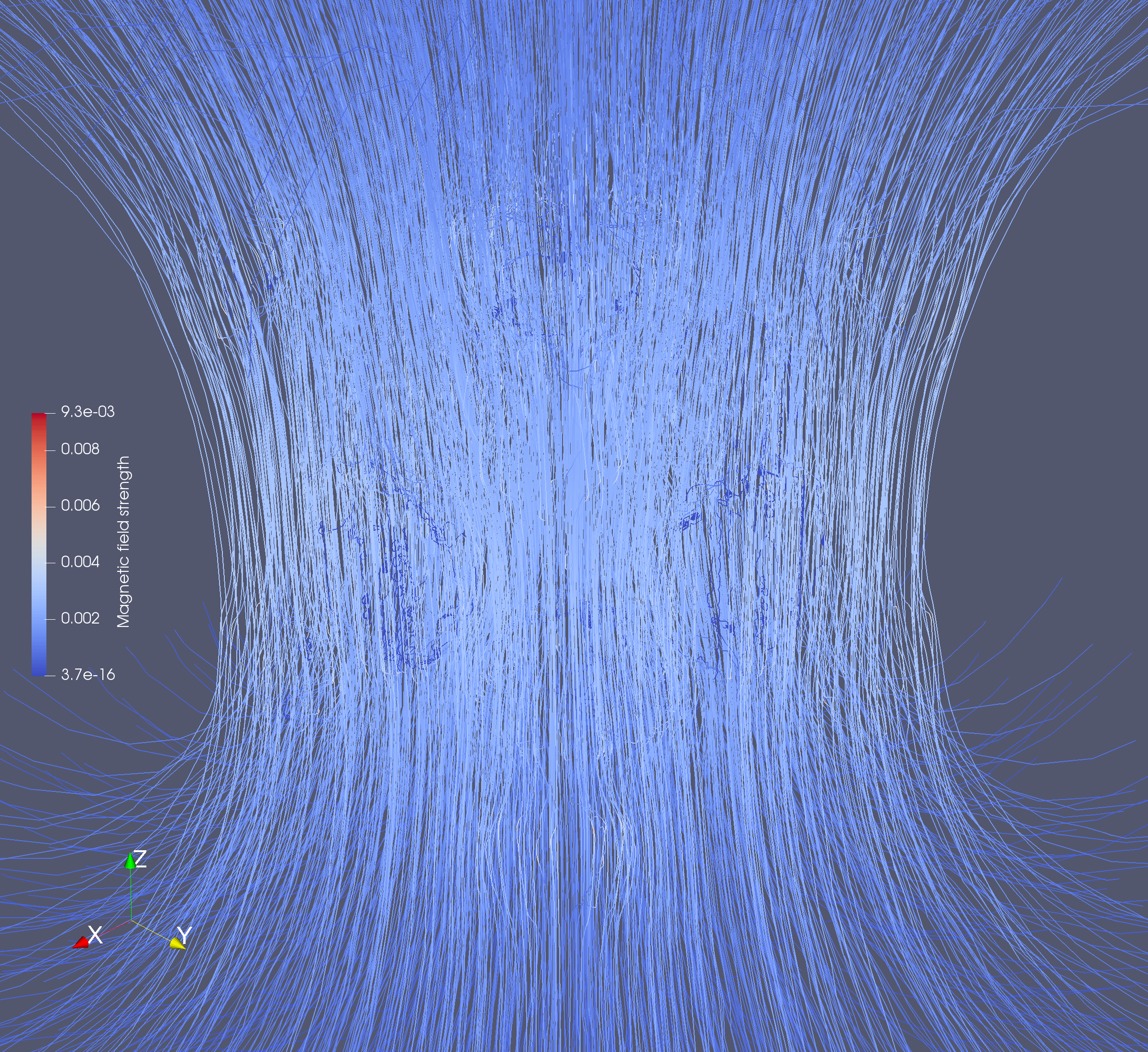}}\,
\subfloat{\label{fig:Ex1:Field_2}\includegraphics[width=0.24\textwidth]{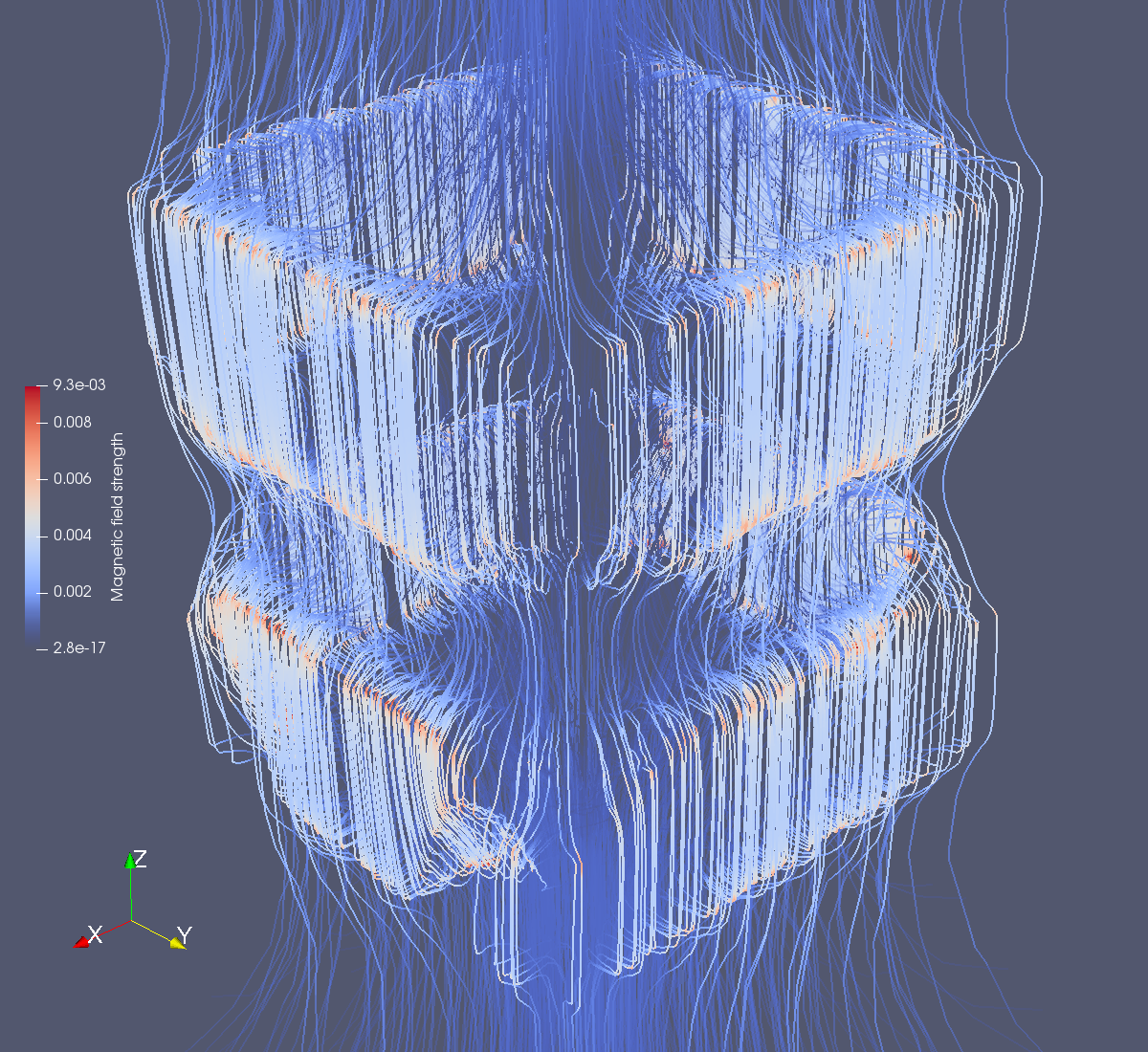}}
\caption{Different views on the magnetic field at the initial and the final iteration. a.)--b.): 2D slice in the center. c.)--d.): Total shot from the same view as \cref{fig:Ex1:Shapes}.} 
\label{fig:Ex1:Field}
\vspace{-0.1cm}
\end{figure}

\subsection{First example} We set $\E_d \equiv 0$ in compliance with \ref{assump:P} to find the optimal shape of a superconductor that minimizes both the electromagnetic field penetration and the volume of material.  
This example is motivated  by the HTS application in the superconducting shielding (cf. \cite{KvitkovicEtAl15}).
We take $\kappa \equiv 8\cdot 10^7$, which is a reasonable choice considering that the electric field strength is roughly $|\E| \approx 10^{-3}$ due to the weak applied current strength $|\f|$. 
The initial shape consists of material attached to the boundary of $B$ (see \cref{fig:Ex1:Shape_0}). 
In \cref{fig:Ex1:Shape_1,fig:Ex1:Shape_2,fig:Ex1:Shape_3} we see some snapshots of the evolving shape generated by our algorithm. 
The algorithm generates two connected components on the top and the bottom of the (lateral) boundary. 
It is interesting to observe that the magnetic field  ($\curl \E$) hits the boundary of the bounding box $B$ from above and, despite the small amount of material used, the field lines do not penetrate through the inside of the area enclosed by the superconductor (see \cref{fig:Ex1:Field_1,fig:Ex1:Field_2}).
Moreover, in \cref{fig:Ex1:Field} we can compare the magnetic field penetration for the initial and the final shape from different camera perspectives. 
The interior of the initial shape is barely protected from penetration, whereas the final shape redirects the magnetic field lines such that they are condensed on the outside of $B$. 

In the final iteration the functional value is around $0.444$ at a volume of roughly $0.278$ which is only $27.8\%$ of the volume of $B$. The E-field fraction in the cost functional amounts roughly to $0.166$. This means that there is only a weak magnetic field left in small areas of $B$. The penetration is mostly between the connected components on the lateral surface of the conducting material.
The development of the functional value as well as the volume fraction is documented in \cref{fig:ValuesEx1} and the minimal value is reached after roughly $125$ iterations. 
Thereafter, it remains almost constant.

We also observe a slight increase of the cost functional at iterations $43$ and $44$, 
due to a topological change in the design. 
Indeed, at iteration $42$ the components on the lateral sides of the cube are disconnected (see \cref{fig:Ex1:Shape_1}), and then merge at iteration $45$ (see \cref{fig:Ex1:Shape_2}). 
This increase of the cost functional due to a topological change is a well-known issue with the level set method; 
see \cite{MR3840889} for a recent study on this issue.
However, in this example the increase in the functional value is negligible and immediately compensated by a sharp decrease.
\begin{figure}[tbhp]
\vspace{-0.4cm}
\centering
 \subfloat{\label{fig:ValuesEx1}\includegraphics[width=0.49\textwidth]{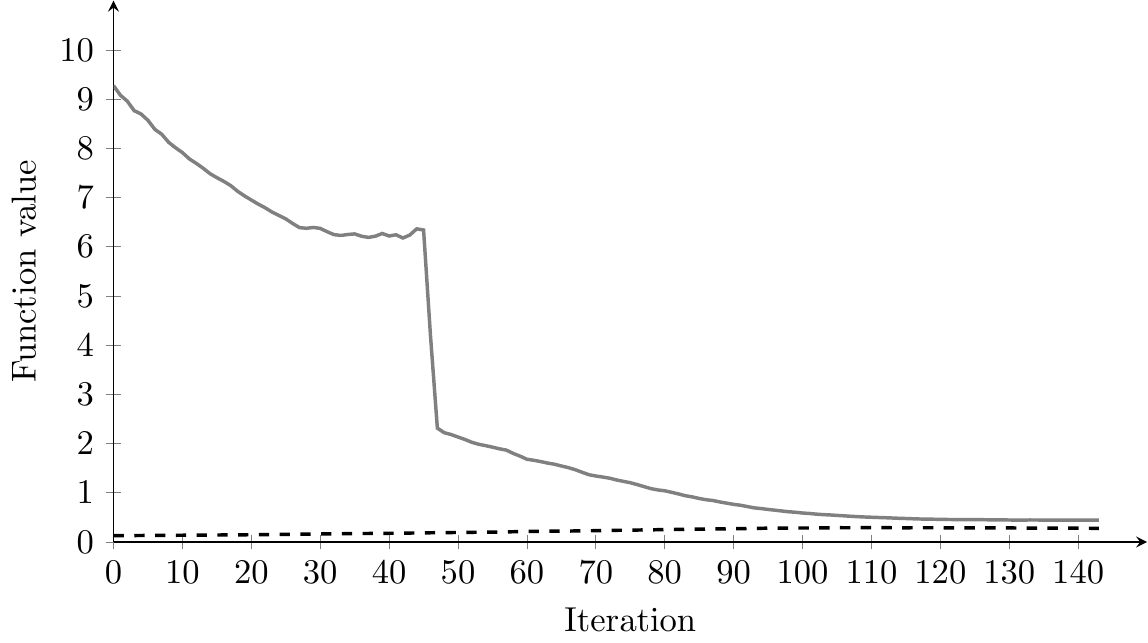}}\,
 \subfloat{\label{fig:ValuesEx2}\includegraphics[width=0.43\textwidth]{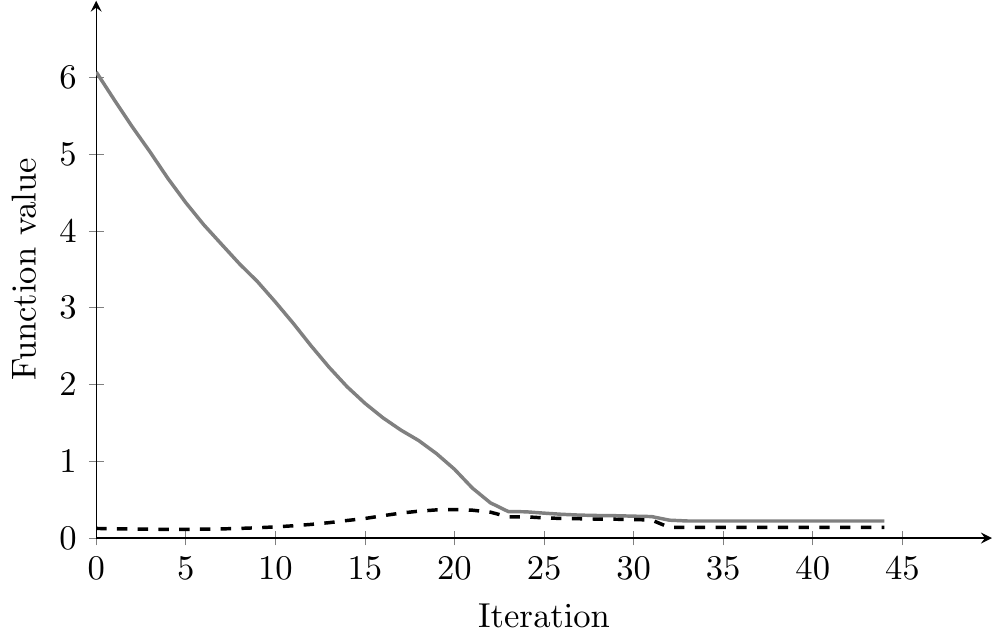}}
  \vspace{-0.3cm}
 \caption{Function value (solid) and volume (dashed): 1. Example (left), 2. Example (right).}
 \vspace{-0.3cm}
\end{figure}
\subsection{Second example}
In our second example, we place a superconducting ball $\om_b$ with radius $r_b = 0.5$ inside $B$ (see \cref{fig:Ex2:Orig_Shape}) and compute $\E_d$ as the corresponding solution of \cref{eq:RegVI}. 
The resulting magnetic field is displayed in \cref{fig:Ex2:Orig_Field_1,fig:Ex2:Orig_Field_2}. We initialized the algorithm with the same parameters and the same initial shape as in the first example (see \cref{fig:Ex1:Shape_0}). 
In the end, we obtain two bell-shaped components  connected by small transitions on the boundary. In \cref{fig:Ex2:Shape_1,fig:Ex2:Shape_2,fig:Ex2:Shape_3} we see this shape from different camera positions. It corresponds to a functional value of $0.223$ where the electric field costs get as low as $0.071$ at a volume fraction of $0.153$.  As the original superconductor was a ball with radius $0.5$,  our algorithm computed  an optimal shape with around $70\%$ less material. 
The development of the functional value and the volume is documented in \cref{fig:ValuesEx2}. Moreover, the descent in this example is smoother and notably faster than the first example. We explain this by the fact that the second choice of $\E_d$ gives more structure than simply $\E_d \equiv 0$. Thus, the algorithm has less possibilities to design the superconductor and converges faster.

\begin{figure}[tbhp]
  \vspace{-0.2cm}
 \subfloat{\label{fig:Ex2:Orig_Shape}\includegraphics[width=0.24\textwidth]{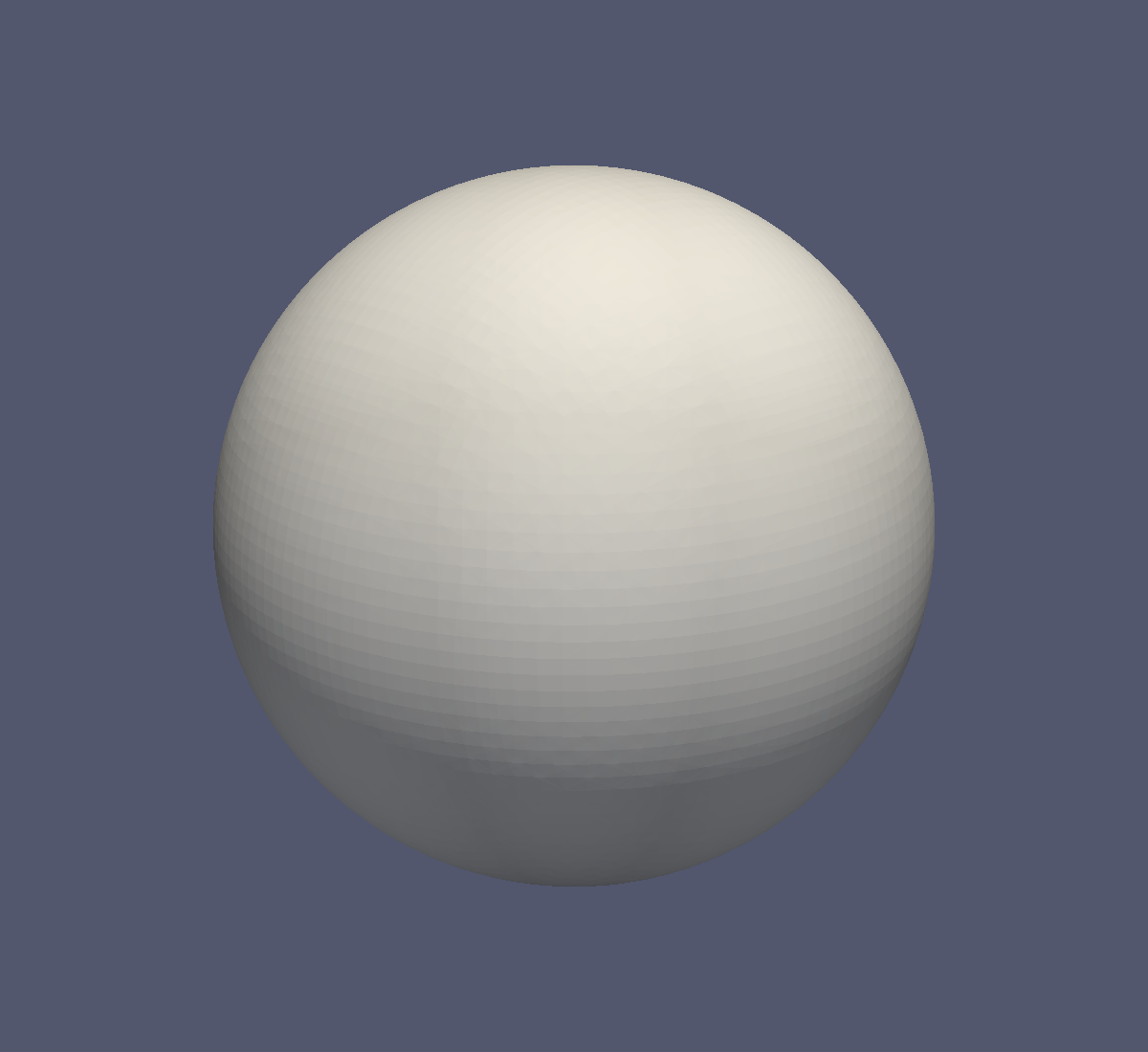}}\, 
 \subfloat{\label{fig:Ex2:Shape_1}\includegraphics[width=0.24\textwidth]{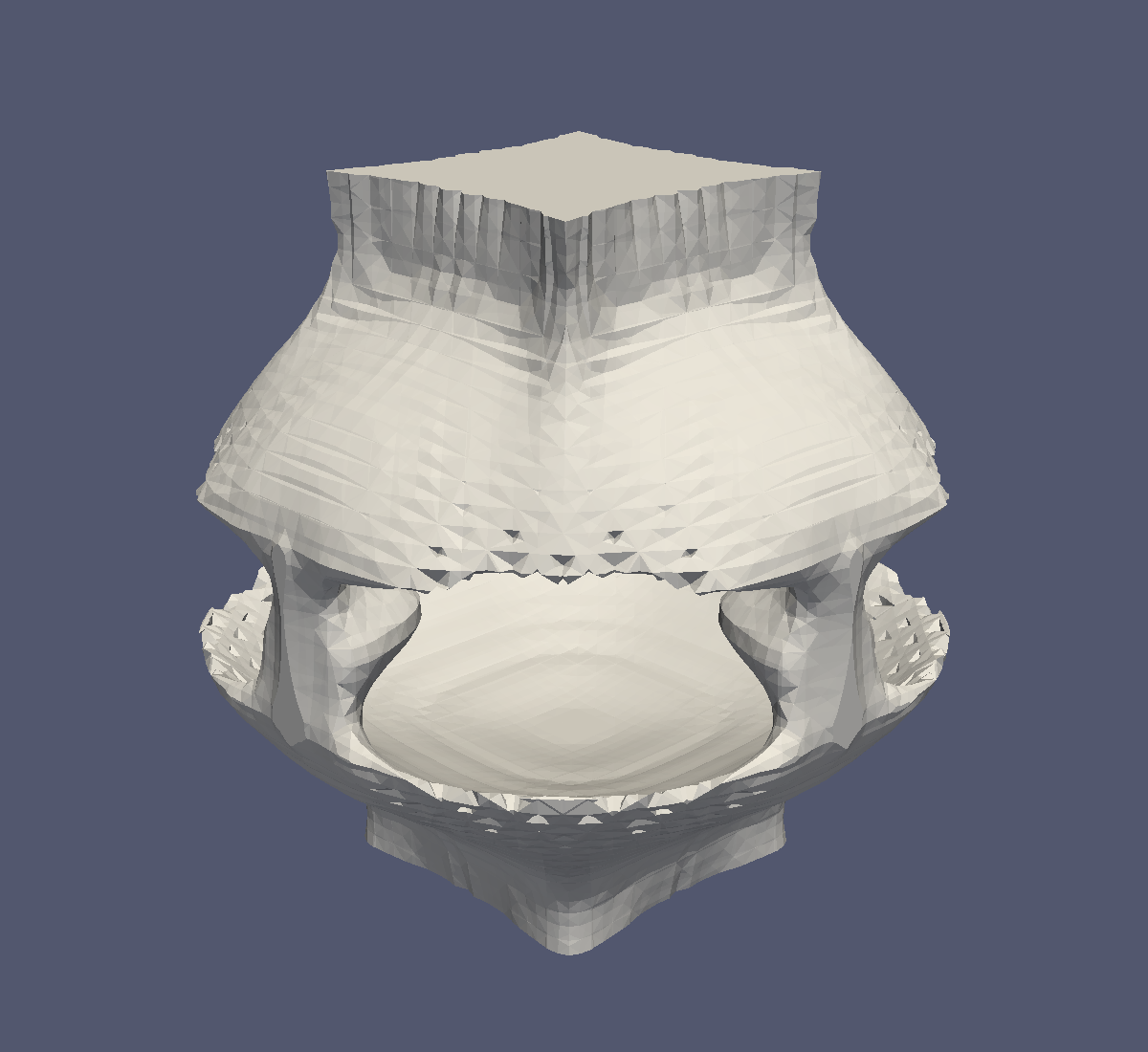}}\,
 \subfloat{\label{fig:Ex2:Shape_2}\includegraphics[width=0.24\textwidth]{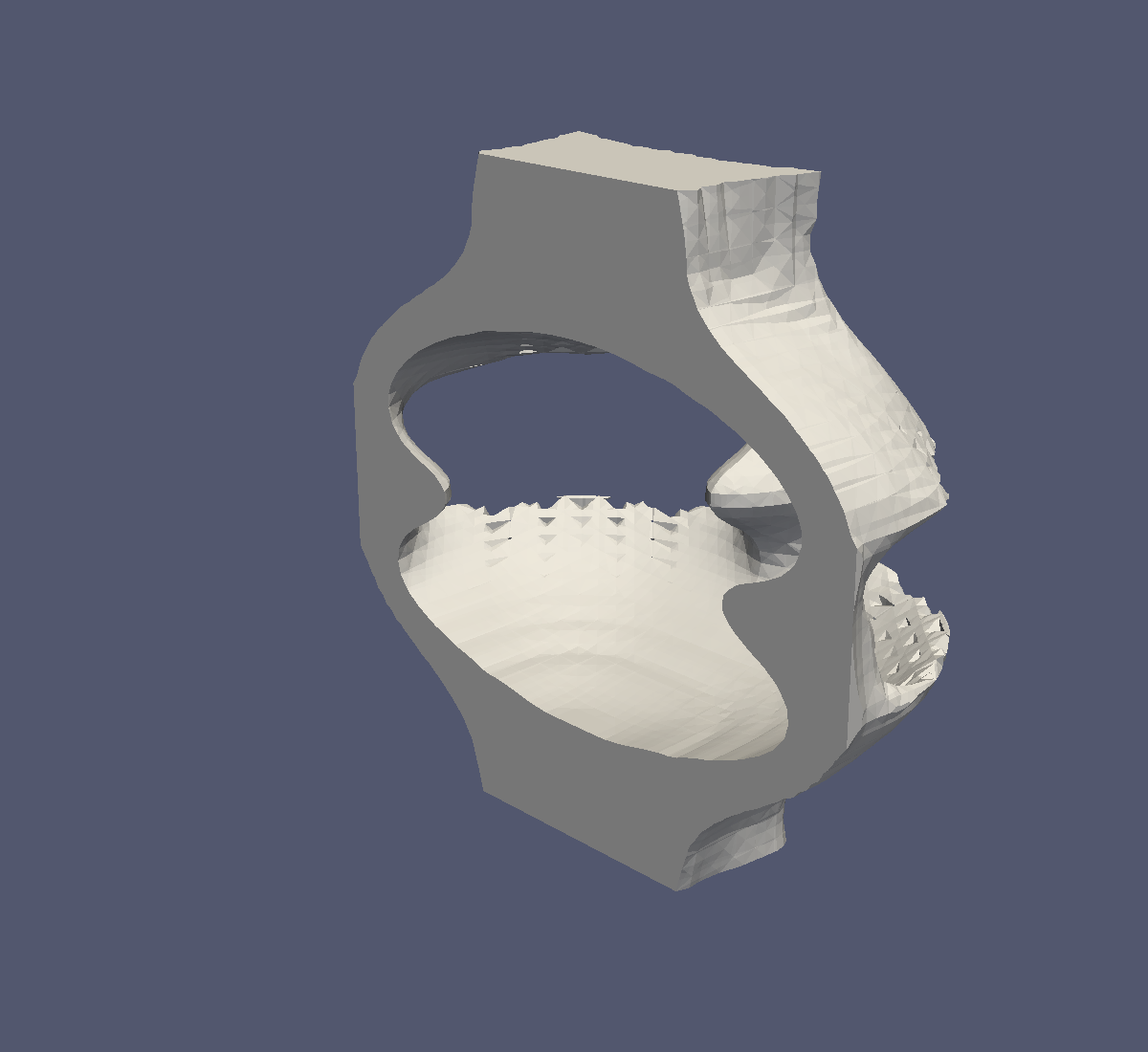}}\,
 \subfloat{\label{fig:Ex2:Shape_3}\includegraphics[width=0.24\textwidth]{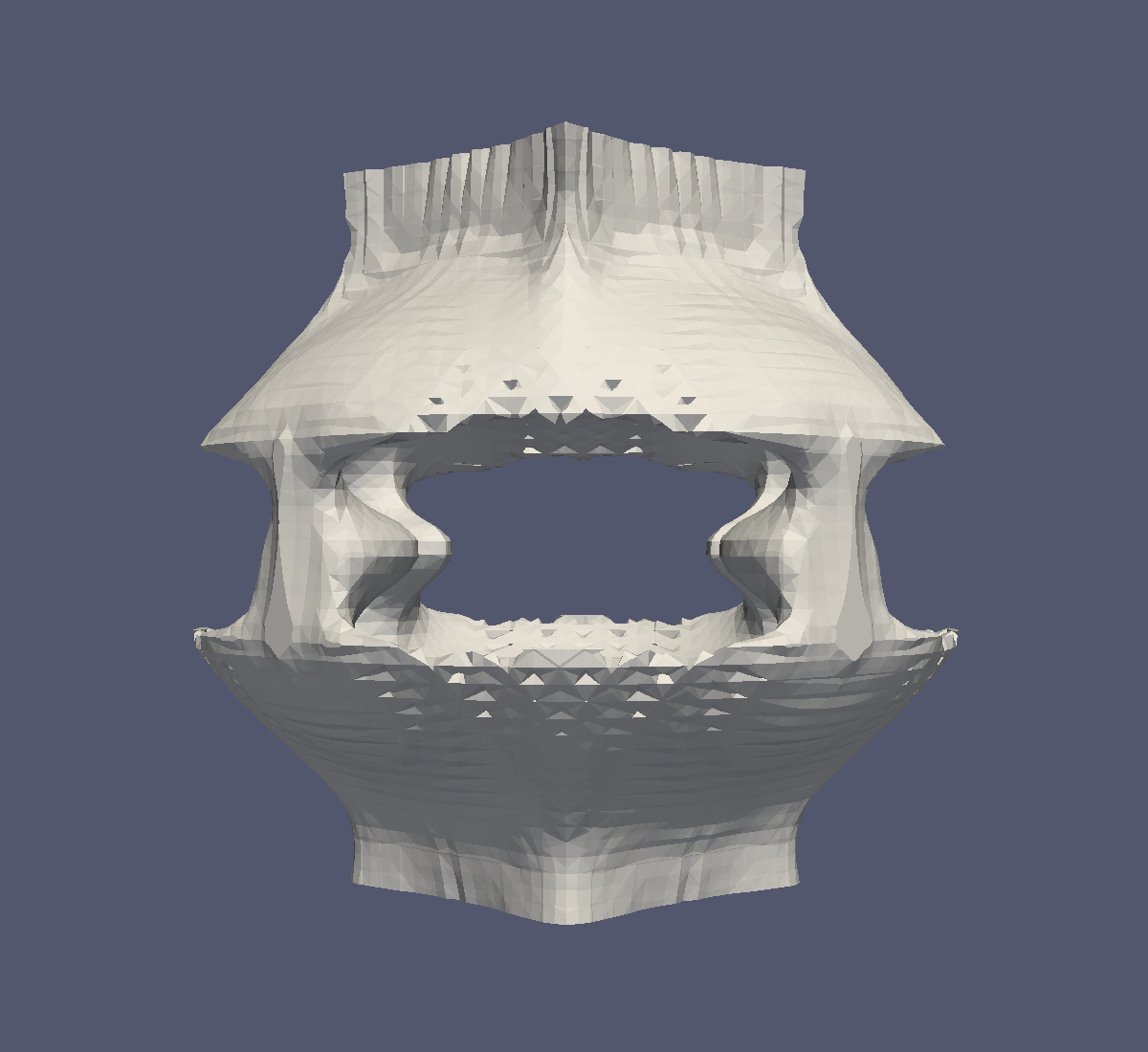}}
 \caption{The original superconductor and the final shape generated by the algorithm in the second example. The third figure is the final shape clipped along the plane $x = 0.5$.} 
 \label{fig:Ex2:Shapes}
 
 \subfloat{\label{fig:Ex2:Orig_Field_1}\includegraphics[width=0.24\textwidth]{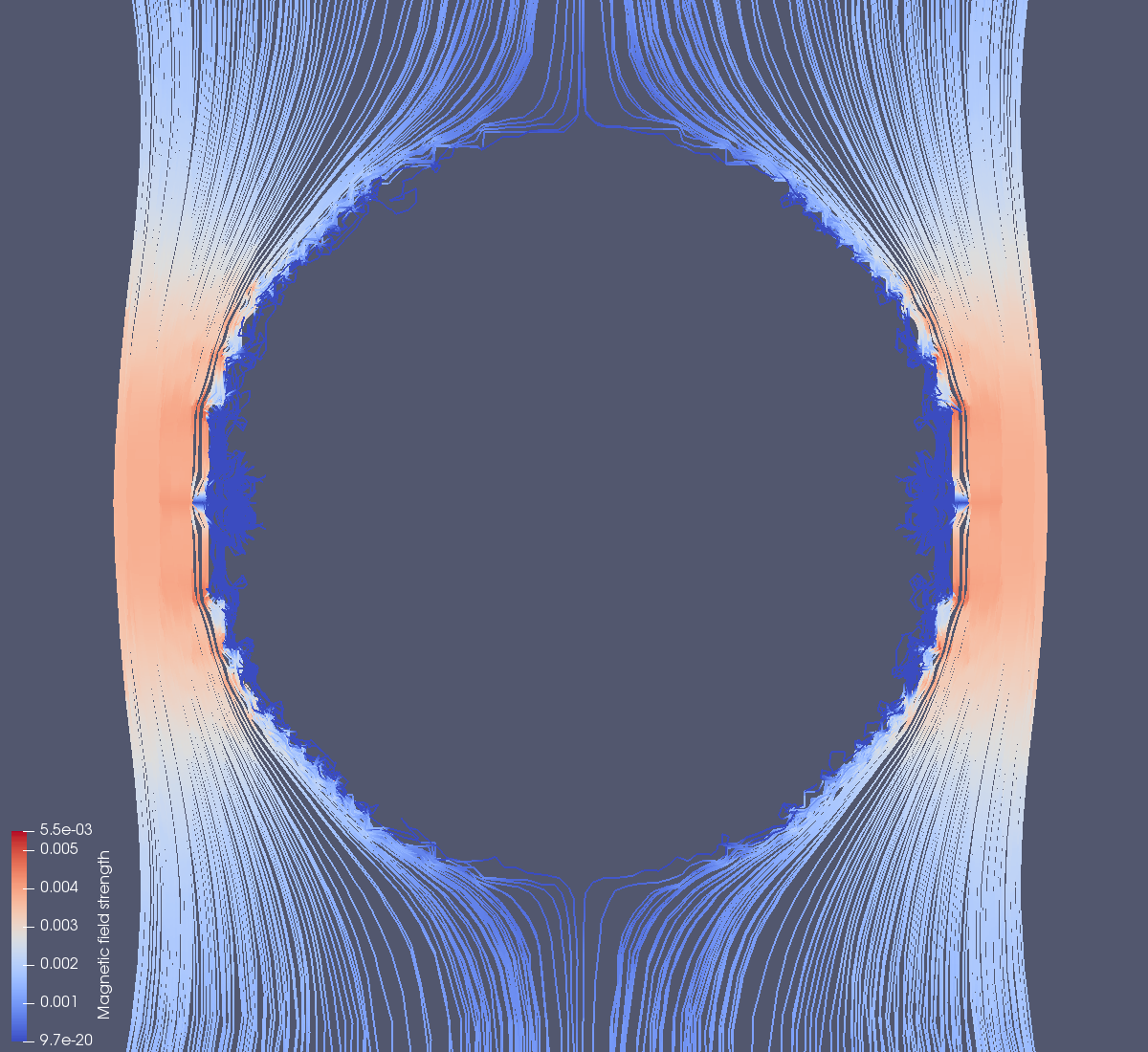}}\, 
 \subfloat{\label{fig:Ex2:Field_1}\includegraphics[width=0.24\textwidth]{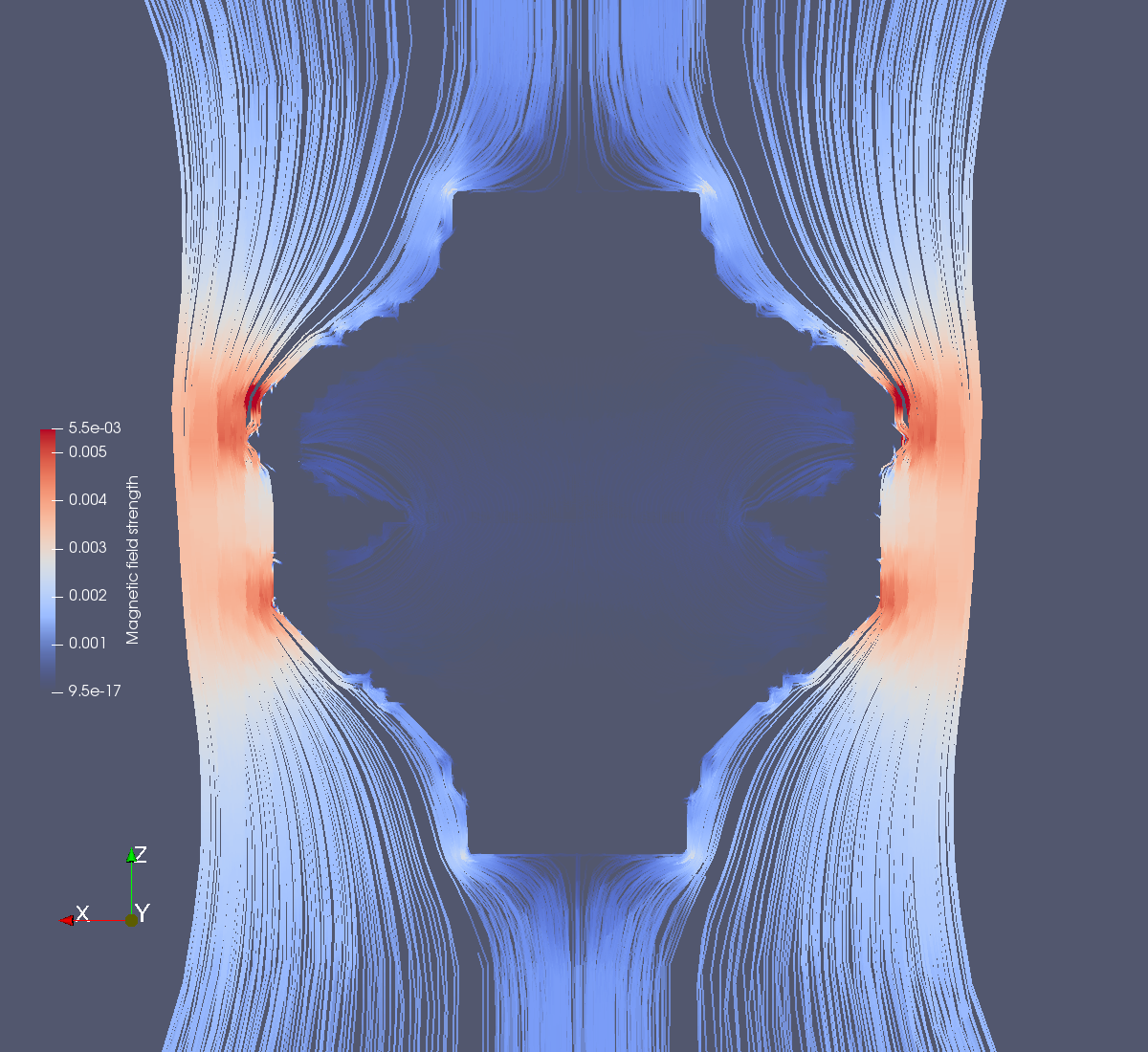}}\,
 \subfloat{\label{fig:Ex2:Orig_Field_2}\includegraphics[width=0.24\textwidth]{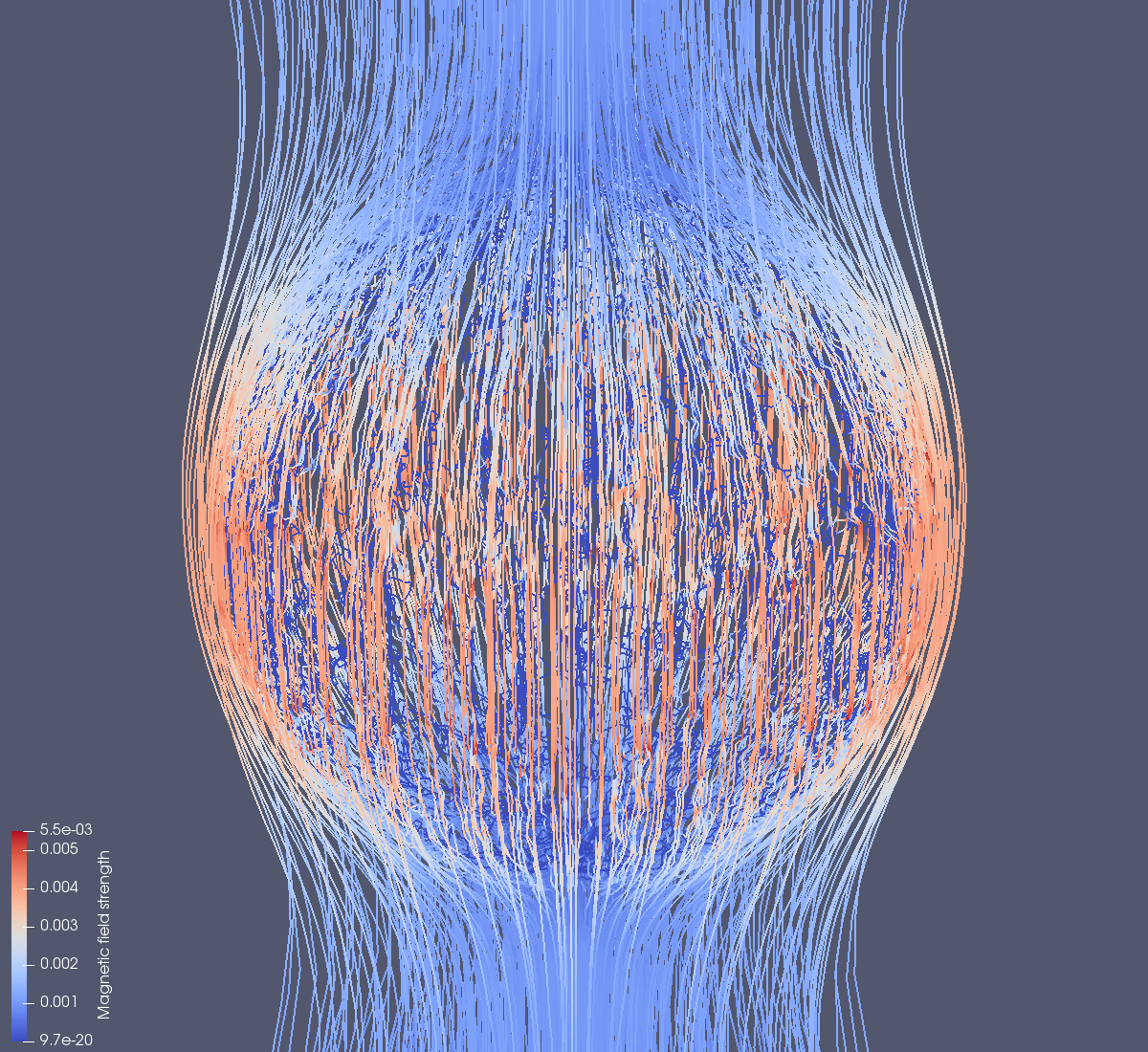}}\,
 \subfloat{\label{fig:Ex2:Field_2}\includegraphics[width=0.24\textwidth]{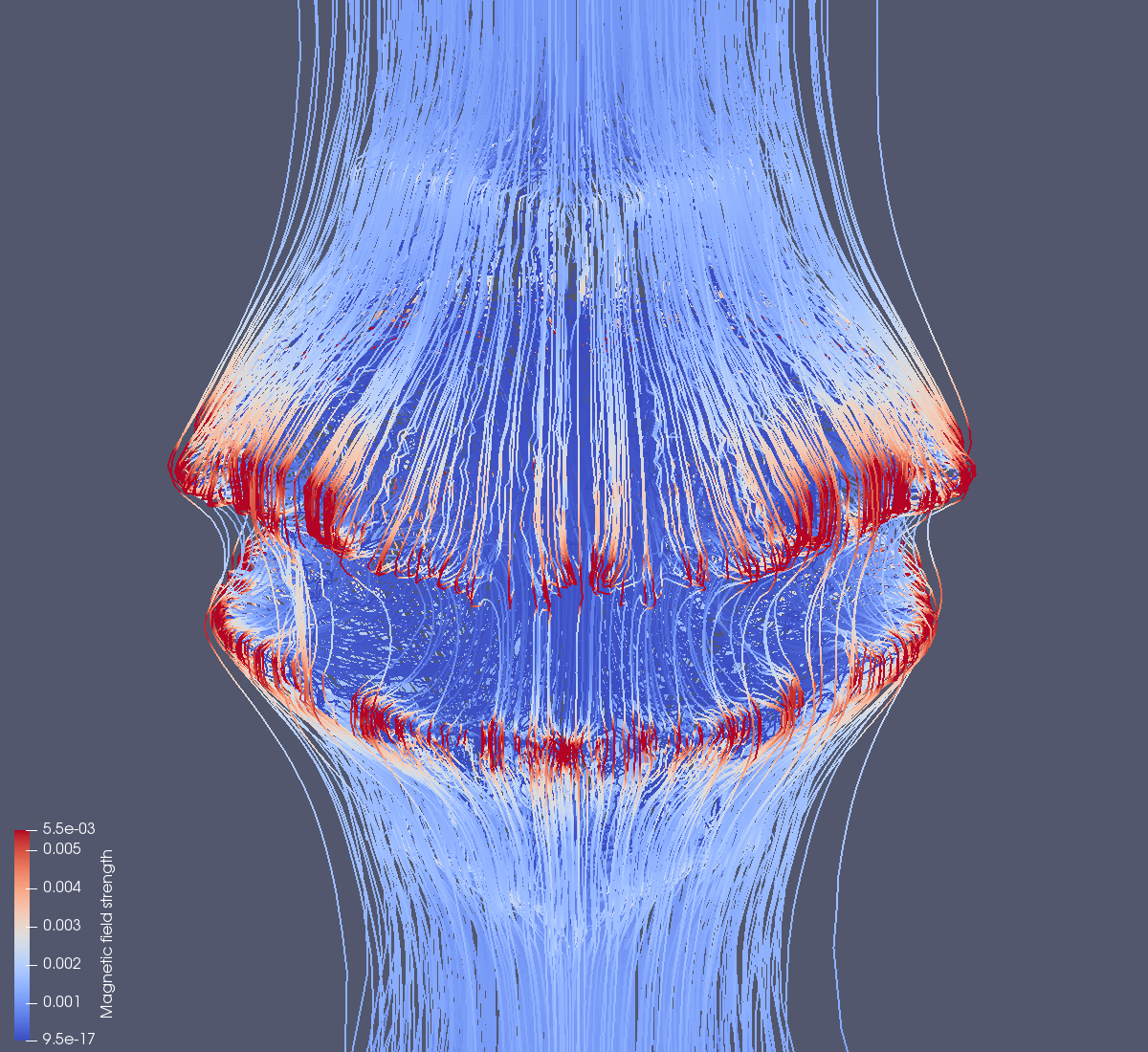}}
 \caption{Different views on the magnetic field of the original and the final superconductor. Left: 2D slice in the center. Right: Total shot from the same view as \cref{fig:Ex2:Orig_Shape,fig:Ex2:Shape_1}.} 
 \label{fig:Ex2:Field}
  \vspace{-0.2cm}
\end{figure}

\subsection{Convergence tests with respect to \texorpdfstring{$\gamma$}{gamma}}\label{sec:convtests}

Let us now report on a numerical test to verify our theoretical convergence result (\cref{thm:ConvergenceEgLamg}). Since no analytical solution is available for the limit case \cref{eq:P}, we compare the numerical results of our algorithm with two different regularization parameters $\hat{\gamma} = 7\cdot10^4$ and $\tilde{\gamma} = 7\cdot10^5$. For these choices, we terminated our algorithm  after $143$ iterations and computed the norm distance between the two numerical solutions: 
\begin{align*}
  \norm{\chi_{\omega^{\hat{\gamma}}} - \chi_{\omega^{\tilde{\gamma}}}}{L^1(\Omega)} \approx 2.88\cdot 10^{-3} \quad \text{and}\quad \Hcurlnorm{\E^{\hat{\gamma}} - \E^{\tilde{\gamma}}} \approx 1.39\cdot 10^{-4}.
\end{align*} 
This relatively small value indicates the   convergence for $\gamma \to \infty$  (\cref{thm:ConvergenceEgLamg}). In particular, we observe that, for sufficiently large penalization parameter $\gamma$, a remarkable change in   $\gamma$ would only lead to a small change in the computed optimal shape.

\subsection{Shape optimization with symmetric design}\label{sec:sym}
In many applications, it is desirable to obtain an optimal design which has certain prescribed symmetries.
These can be, for instance, the consequence of symmetries of the geometry and the data that imply symmetries in the continuous solution. 
However, in practice, the numerically optimized design may deviate substantially from these symmetries, usually due to a non-symmetric discretization. 
This can be mitigated by refining the discretization which may not always be an affordable option, especially for 3D problems.
Thus, imposing the symmetry as a constraint for the discretized problem can be a valuable alternative.

In this section we describe a method to obtain a descent direction for our minimization algorithm for \cref{eq:Preg} while imposing a symmetry constraint.
Therefore, we assume $B\subset\R^3$ (cf. \ref{assump:P}) to be additionally symmetric with respect to some plane $Q\subset \R^3$. 
Without loss of generality, we may assume that $Q = \{x \in \R^3 \, | \, x_3=0\}$. 
Thanks to \cref{thm:ShapeDerivative}, the shape derivative of $J_\gamma(\om)$ exists for every $\om\in\mathcal O$ and admits the following tensor expression (see \cref{eq:ShapeDerivative})
\begin{equation*}
dJ_\gamma(\om)(\VV) =   \int_{B} S_1 : D\VV + \bS_0\cdot\VV\, dx\quad \forall \VV\in \bm{\mathcal C}_c^{0,1}(\Omega) \textrm{ with }  \operatorname{supp} \VV \subset\subset B.
\end{equation*}
Now, a descent direction for $J_\gamma$ can be found by computing a solution $\widehat\VV \in \V_h$ of
\begin{equation*}
\mathcal{B}(\widehat\VV,\bzeta) =  - dJ_\gamma(\om)(\bzeta) = -  \int_{B} S_1 : D\bzeta + \bS_0\cdot\bzeta\, dx,\quad  \forall \bzeta\in \V_h,
\end{equation*}
where $\mathcal{B}$ is a positive definite bilinear form on $\V_h \times \V_h$ (see \cref{VP_1}). 
The descent direction $\widehat\VV \neq 0$ is not necessarily symmetric with respect to $Q$.
Our aim now is to construct a symmetric descent direction out of $\widehat\VV$.  
Therefore, we denote the reflection with respect to the plane $Q$ by $R_Q: \R^3\to \R^3$ which is given by $(x_1,x_2,x_3)^\transp \mapsto (x_1,x_2, -x_3)^\transp$. We choose an appropriate triangulation of $B$ such that  the corresponding $\mathbb P_1$-finite element space $\V_h$  satisfies 
\begin{equation}\label{P1}
\bzeta \in \V_h \quad \Rightarrow \quad \bzeta \circ R_Q \in \V_h.
\end{equation}
Clearly, a vector field  $\VV = (\theta_1,\theta_2,\theta_3)^\transp:\R^3\to\R^3$ is symmetric with respect to $Q$ if and only if
\begin{equation}\label{DRQ}
\VV\circ R_Q(x) = (\theta_1(x), \theta_2(x), -\theta_3(x))^\transp = DR_Q \VV(x) \quad \forall x \in \R^3.
\end{equation}
We define the vector field
\begin{equation*}
\VV := \widehat\VV + DR_Q\widehat\VV\circ R_Q 
\end{equation*}
which is indeed symmetric with respect to $Q$. Due to $R_Q^{-1} = R_Q$ and $DR_Q^{-1} = DR_Q$, we readily obtain that \cref{DRQ} holds for $\VV$ by calculating
\begin{align*}
 \VV\circ R_Q = \widehat\VV\circ R_Q + DR_Q\widehat\VV = DR_Q\VV. 
\end{align*}
Next, we will prove that $\VV$ also provides a descent direction. In fact, the bilinear form $\mathcal B$ that was used for our numerical experiments \cref{VP_2} consists of three summands. However, as the arguments are virtually the same for all of them, we will only focus on the first one, i.e.,
\begin{equation*}
\widetilde{\mathcal{B}}: \V_h \times \V_h  \to \R,\quad (\boeta,\bzeta)  \mapsto  \int_B D\boeta : D\bzeta\,dx. 
\end{equation*}
Since $\widehat\VV\in \V_h$,  we have due to \cref{P1}  that $\VV\in \V_h$, and therefore
\begin{align}\label{eq:symm1}
  \widetilde{\mathcal{B}}(\widehat\VV,\VV)
&=   \int_B D\widehat\VV : D(\widehat\VV  + DR_Q\widehat\VV\circ R_Q)\,dx\\\notag
& =   \int_{B} D\widehat\VV : [ D\widehat\VV  + DR_Q D(\widehat\VV\circ R_Q)  ]\,dx\\\notag
& =   \int_{B} D\widehat\VV : [ D\widehat\VV  + DR_Q (D\widehat\VV\circ R_Q) DR_Q  ]\,dx.
\end{align}
In order to exploit the symmetry properties of $B$, we introduce half-sets $B^+ =B \cap \{x_3>0\}$ and $B^- =B \cap \{x_3<0\}$. Thus, we may split the integral in \cref{eq:symm1} and apply the change of variables $x\mapsto R_Q(x)$ in the integral over $B^-$. Therefore, using the fact that $DR_Q = DR_Q^{-1} =  DR_Q^\transp$ we finally obtain
\begin{align*}
\widetilde{\mathcal{B}}(\widehat\VV,\VV)  & =   \int_{B^+} D\widehat\VV : [ D\widehat\VV  + DR_Q (D\widehat\VV\circ R_Q) DR_Q  ]\,dx\\
&\quad  +  \int_{B^+} D\widehat\VV \circ R_Q: [ D\widehat\VV\circ R_Q  + DR_Q D\widehat\VV DR_Q  ]\,dx\\
& =   \int_{B^+} D\widehat\VV : [ D\widehat\VV  + DR_Q (D\widehat\VV\circ R_Q) DR_Q  ]\,dx\\
&\quad  + \int_{B^+} DR_Q  (D\widehat\VV \circ R_Q)  DR_Q: [DR_Q (D\widehat\VV\circ R_Q) DR_Q + D\widehat\VV   ]\,dx\\
& =   \int_{B^+} |D\widehat\VV  + DR_Q (D\widehat\VV\circ R_Q) DR_Q|^2\,dx >0 .
\end{align*}
Similar calculations yield $dJ_\gamma (\omega) (\VV) =-  \mathcal{B}(\widehat\VV,\VV) <0$.
Thus, $\VV$ is a descent direction for $J_\gamma$ that satisfies the symmetry property \cref{DRQ}.
Using $\VV$ instead of $\widehat\VV$ in our numerical algorithm yields an optimized design that is symmetric with respect to $Q$.

Finally, observe that if two symmetries with respect to two orthogonal planes $Q_1$ and $Q_2$ are desired, applying the symmetrization process described above  first with respect to $Q_1$ and then with respect to $Q_2$ will yield the desired symmetries for $\VV$.

\vspace{-0.25cm}
\bibliographystyle{abbrv} 
\bibliography{litbank}

\begin{thebibliography}{10}

\bibitem{MR2033390}
G.~Allaire, F.~Jouve, and A.-M. Toader.
\newblock Structural optimization using sensitivity analysis and a level-set
  method.
\newblock {\em J. Comput. Phys.}, 194(1):363--393, 2004.

\bibitem{BarbuFriedman1991}
V.~Barbu and A.~Friedman.
\newblock Optimal design of domains with free-boundary problems.
\newblock {\em SIAM Journal on Control and Optimization}, 29(3):623--637, 1991.

\bibitem{barpri15}
J.~W. Barrett and L.~Prigozhin.
\newblock Sandpiles and superconductors: nonconforming linear finite element
  approximations for mixed formulations of quasi-variational inequalities.
\newblock {\em IMA J. Numer. Anal.}, 35(1):1--38, 2015.

\bibitem{bos94}
A.~Bossavit.
\newblock Numerical modelling of superconductors in three dimensions: a model
  and a finite element method.
\newblock {\em IEEE Transactions on Magnetics}, 30(5):3363--3366, 1994.

\bibitem{MR862783}
J.~C\'{e}a.
\newblock Conception optimale ou identification de formes: calcul rapide de la
  d\'{e}riv\'{e}e directionnelle de la fonction co\^{u}t.
\newblock {\em RAIRO Mod\'{e}l. Math. Anal. Num\'{e}r.}, 20(3):371--402, 1986.

\bibitem{DeLosReyes11}
J.~De~Los~Reyes.
\newblock Optimal control of a class of variational inequalities of the second
  kind.
\newblock {\em SIAM Journal on Control and Optimization}, 49(4):1629--1658,
  2011.

\bibitem{DelfourZolesio11}
M.~Delfour and J.~Zol{\'e}sio.
\newblock {\em Shapes and Geometries}.
\newblock Society for Industrial and Applied Mathematics, second edition, 2011.

\bibitem{MR1728769}
Z.~Denkowski and S.~Mig\'{o}rski.
\newblock Optimal shape design for elliptic hemivariational inequalities in
  nonlinear elasticity.
\newblock In {\em Variational calculus, optimal control and applications},
  volume 124 of {\em Internat. Ser. Numer. Math.}, pages 31--40.
  Birkh\"{a}user, Basel, 1998.

\bibitem{MR1606879}
Z.~Denkowski and S.~Mig\'{o}rski.
\newblock Optimal shape design problems for a class of systems described by
  hemivariational inequalities.
\newblock {\em J. Global Optim.}, 12(1):37--59, 1998.

\bibitem{ellkas07}
C.~M. Elliott and Y.~Kashima.
\newblock A finite-element analysis of critical-state models for type-{II}
  superconductivity in 3{D}.
\newblock {\em IMA J. Numer. Anal.}, 27(2):293--331, 2007.

\bibitem{MR2523581}
G.~Fr\'{e}miot, W.~Horn, A.~Laurain, M.~Rao, and J.~Soko{\l}owski.
\newblock On the analysis of boundary value problems in nonsmooth domains.
\newblock {\em Dissertationes Math.}, 462:149, 2009.

\bibitem{MR3878790}
B.~F\"{u}hr, V.~Schulz, and K.~Welker.
\newblock Shape optimization for interface identification with obstacle
  problems.
\newblock {\em Vietnam J. Math.}, 46(4):967--985, 2018.

\bibitem{MR2356899}
P.~Fulma\'{n}ski, A.~Laurain, J.-F. Scheid, and J.~Soko{\l}owski.
\newblock A level set method in shape and topology optimization for variational
  inequalities.
\newblock {\em Int. J. Appl. Math. Comput. Sci.}, 17(3):413--430, 2007.

\bibitem{MR3562369}
C.~Heinemann and K.~Sturm.
\newblock Shape optimization for a class of semilinear variational inequalities
  with applications to damage models.
\newblock {\em SIAM J. Math. Anal.}, 48(5):3579--3617, 2016.

\bibitem{MR3791463}
A.~Henrot and M.~Pierre.
\newblock {\em Shape variation and optimization}, volume~28 of {\em EMS Tracts
  in Mathematics}.
\newblock European Mathematical Society (EMS), Z\"{u}rich, 2018.

\bibitem{HintLaur11}
M.~Hinterm{\"u}ller and A.~Laurain.
\newblock Optimal shape design subject to elliptic variational inequalities.
\newblock {\em SIAM Journal on Control and Optimization}, 49(3):1015--1047,
  2011.

\bibitem{MR3350625}
M.~Hinterm{\"u}ller, A.~Laurain, and I.~Yousept.
\newblock Shape sensitivities for an inverse problem in magnetic induction
  tomography based on the eddy current model.
\newblock {\em Inverse Problems}, 31(6):065006, 25, 2015.

\bibitem{MR2434064}
K.~Ito, K.~Kunisch, and G.~H. Peichl.
\newblock Variational approach to shape derivatives.
\newblock {\em ESAIM Control Optim. Calc. Var.}, 14(3):517--539, 2008.

\bibitem{MR1662168}
F.~Jochmann.
\newblock The semistatic limit for {M}axwell's equations in an exterior domain.
\newblock {\em Comm. Partial Differential Equations}, 23(11-12):2035--2076,
  1998.

\bibitem{MR2977497}
H.~Kasumba and K.~Kunisch.
\newblock On shape sensitivity analysis of the cost functional without shape
  sensitivity of the state variable.
\newblock {\em Control Cybernet.}, 40(4):989--1017, 2011.

\bibitem{MR1294835}
M.~Ko{\v c}vara and J.~V. Outrata.
\newblock Shape optimization of elastoplastic bodies governed by variational
  inequalities.
\newblock In {\em Boundary control and variation ({S}ophia {A}ntipolis, 1992)},
  volume 163 of {\em Lecture Notes in Pure and Appl. Math.}, pages 261--271.
  Dekker, New York, 1994.

\bibitem{KvitkovicEtAl15}
J.~Kvitkovic, D.~Davis, M.~Zhang, and S.~Pamidi.
\newblock Magnetic shielding characteristics of second generation high
  temperature superconductors at variable temperatures obtained by cryogenic
  helium gas circulation.
\newblock {\em IEEE Trans. Appl. Supercond.}, 25(3), 6 2015.

\bibitem{MR3840889}
A.~Laurain.
\newblock Analyzing smooth and singular domain perturbations in level set
  methods.
\newblock {\em SIAM J. Math. Anal.}, 50(4):4327--4370, 2018.

\bibitem{Laurain2018}
A.~Laurain.
\newblock A level set-based structural optimization code using fenics.
\newblock {\em Structural and Multidisciplinary Optimization},
  58(3):1311--1334, Sep 2018.

\bibitem{MR3535238}
A.~Laurain and K.~Sturm.
\newblock Distributed shape derivative {\it via} averaged adjoint method and
  applications.
\newblock {\em ESAIM Math. Model. Numer. Anal.}, 50(4):1241--1267, 2016.

\bibitem{LionsStampacchia67}
J.~L. Lions and G.~Stampacchia.
\newblock Variational inequalities.
\newblock {\em Communications on Pure and Applied Mathematics}, 20(3):493--519,
  1967.

\bibitem{MR1106360}
W.~B. Liu and J.~E. Rubio.
\newblock Optimal shape design for systems governed by variational
  inequalities. {I}. {E}xistence theory for the elliptic case.
\newblock {\em JOTA}, 69(2):351--371, 1991.

\bibitem{fenics:book}
A.~Logg, K.-A. Mardal, and G.~N. Wells, editors.
\newblock {\em Automated Solution of Differential Equations by the Finite
  Element Method}, volume~84 of {\em Lect. Notes Comput. Sci. Eng.}
\newblock Springer, 2012.

\bibitem{Monk03}
P.~Monk.
\newblock {\em Finite Element Methods for {M}axwell's Equations}.
\newblock Numerical Analysis and Scientic Computation. Clarendon Press, 2003.

\bibitem{MR1816853}
A.~My{\'s}li{\'n}ski.
\newblock Domain optimization for unilateral problems by an embedding domain
  method.
\newblock In {\em Shape optimization and optimal design ({C}ambridge, 1999)},
  volume 216 of {\em Lect. Pure Appl. Math.}, pages 355--370. Dekker, New York,
  2001.

\bibitem{ned80}
J.-C. N{\'e}d{\'e}lec.
\newblock Mixed finite elements in {${\bf R}^{3}$}.
\newblock {\em Numer. Math.}, 35(3):315--341, 1980.

\bibitem{NeitSokoZole1988}
P.~Neittaanm{\"a}ki, J.~Soko{\l}owski, and J.~P. Zolesio.
\newblock Optimization of the domain in elliptic variational inequalities.
\newblock {\em Applied Mathematics and Optimization}, 18(1):85--98, Jul 1988.

\bibitem{MR2166150}
O.~Pantz.
\newblock Sensibilit\'{e} de l'\'{e}quation de la chaleur aux sauts de
  conductivit\'{e}.
\newblock {\em C. R. Math. Acad. Sci. Paris}, 341(5):333--337, 2005.

\bibitem{pri96}
L.~Prigozhin.
\newblock On the {B}ean critical-state model in superconductivity.
\newblock {\em European J. Appl. Math.}, 7(3):237--247, 1996.

\bibitem{qi2017transposes}
L.~Qi.
\newblock Transposes, {L}-eigenvalues and invariants of third order tensors,
  2017.

\bibitem{Roubicek13}
T.~Roub{\'\i}cek.
\newblock {\em Nonlinear Partial Differential Equations with Applications}.
\newblock International Series of Numerical Mathematics. Springer Basel, 2013.

\bibitem{MR2206676}
J.~Soko{\l}owski and A.~\.{Z}ochowski.
\newblock Modelling of topological derivatives for contact problems.
\newblock {\em Numer. Math.}, 102(1):145--179, 2005.

\bibitem{MR1215733}
J.~Soko{\l}owski and J.-P. Zol{\'e}sio.
\newblock {\em Introduction to shape optimization}, volume~16 of {\em Springer
  Series in Computational Mathematics}.
\newblock Springer-Verlag, Berlin, 1992.

\bibitem{MR3374631}
K.~Sturm.
\newblock Minimax {L}agrangian approach to the differentiability of nonlinear
  {PDE} constrained shape functions without saddle point assumption.
\newblock {\em SICON}, 53(4):2017--2039, 2015.

\bibitem{GlowinskiLionsTremolieres81}
R.~Tr{\'e}moli{\`e}res, J.~L. Lions, and R.~Glowinski.
\newblock {\em Numerical Analysis of Variational Inequalities}.
\newblock Studies in Mathematics and its Applications. Elsevier Science, 1981.

\bibitem{troyou12}
F.~Tr{\"o}ltzsch and I.~Yousept.
\newblock P{DE}-constrained optimization of time-dependent 3{D} electromagnetic
  induction heating by alternating voltages.
\newblock {\em ESAIM Math. Model. Numer. Anal.}, 46(4):709--729, 2012.

\bibitem{WincklerYousept18}
M.~Winckler and I.~Yousept.
\newblock Fully discrete scheme for {B}ean's critical-state model with
  temperature effects in superconductivity.
\newblock  SIAM J. Numer. Anal., 57(6): 2685--2706, 2019.

\bibitem{Yousept19Hyp}
I.~Yousept.
\newblock Hyperbolic {M}axwell variational inequalities of the second kind.
\newblock {\em ESAIM: COCV}, 26, Paper No. 34, 2020.

\bibitem{you12a}
I.~Yousept.
\newblock Optimal control of {M}axwell's equations with regularized state
  constraints.
\newblock {\em Computational Optimization and Applications}, 52(2):559--581,
  2012.

\bibitem{you13}
I.~Yousept.
\newblock Optimal {C}ontrol of {Q}uasilinear {$\boldsymbol
  H(\mathbf{curl})$}-{E}lliptic {P}artial {D}ifferential {E}quations in
  {M}agnetostatic {F}ield {P}roblems.
\newblock {\em SIAM J. Control Optim.}, 51(5):3624--3651, 2013.

\bibitem{Yousept17Hyp}
I.~Yousept.
\newblock Hyperbolic {M}axwell variational inequalities for {B}ean's
  critical-state model in type-{II} superconductivity.
\newblock {\em SIAM J. Numer. Anal.}, 55(5):2444--2464, 2017.

\bibitem{Yousept17OptContr}
I.~Yousept.
\newblock Optimal control of non-smooth hyperbolic evolution {M}axwell
  equations in type-{II} superconductivity.
\newblock {\em SIAM J. Control Optim.}, 55(4):2305--2332, 2017.

\end{thebibliography}


\end{document}